\documentclass{amsart}

\usepackage{amssymb}
\usepackage{thmtools}
\usepackage{hyperref,cleveref}
\usepackage[shortlabels]{enumitem}

\declaretheorem[style=definition,numberwithin=section]{definition}
\declaretheorem[style=plain,numberwithin=section,]{theorem}
\declaretheorem[style=remark,numberwithin=section]{remark}
\declaretheorem[style=plain,numberwithin=section]{lemma}
\declaretheorem[style=plain,numberwithin=section]{corollary}
\declaretheorem[numbered=no,name=Theorem A]{thma}
\declaretheorem[numbered=no,name=Theorem B]{thmb}
\declaretheorem[numbered=no,name=Assumption,refname={assumption su},Refname={assumption su}]{ass}
\declaretheorem[numbered=no,name=Q-Criterion]{qcrit}
\declaretheorem[numbered=no,name=Local Ergodic Theorem]{local}

\numberwithin{equation}{section}

\makeatletter
\newcommand{\mylabel}[2]{#2\def\@currentlabel{#2}\label{#1}}
\makeatother

\begin{document}

\title[]{A conditional proof of the non-contraction property for $N$ falling balls}

\author{Michael Hofbauer-Tsiflakos}
\address{Michael Hofbauer-Tsiflakos}
\curraddr{UAB \\
University Hall 4037 \\
1402 10th Ave S\\
35294 Birmingham, AL}
\email{hofbauer@uab.edu}
\urladdr{https://people.cas.uab.edu/~hofbauer/}





\begin{abstract}
Wojtkowski's system of $N$, $N \geq 2$, falling balls is a nonuniformly hyperbolic smooth dynamical system with singularities. 
It is still an open question whether this system is ergodic. We contribute toward an affirmative answer, by proving 
the non-contraction property, conditioned by the assumption of strict unboundedness. For a certain mass ratio 
the configuration space can be unfolded to a billiard table where the daunting proper alignment condition is satisfied. 
We prove, that the aforementioned unfolded system with three degrees of freedom is ergodic.
\end{abstract}

\maketitle
\tableofcontents

\section{Introduction}

In \cite{W90a,W90b}, Maciej P. Wojtkowski introduced the system of $N$, $N \geq 2$, falling balls. It describes 
the motion of $N$ point masses moving up and down a vertical line, colliding with each other elastically and 
the lowest point mass collides with a rigid floor placed at height zero. The system has $N$ degrees of freedom, 
the positions $q_1,\ldots,q_N$ and the momenta $p_1,\ldots,p_N$. The point masses are placed on top of each 
satisfying $0 \leq q_1 \leq \ldots \leq q_N$. The overall standing assumption 
on the masses is $m_1 > \ldots > m_N$.
Movement occurs due to kinetic energy and a linear potential field on a compact energy surface $E_c$ 
given by the Hamiltonian $H(q,p) = \sum_{i = 1}^N p_i^2 / 2m_i + m_iq_i$. The dynamics are further reduced to 
the Poincar\'e section $\mathcal{M}$ containing the states right after a collision of two point masses or a 
collision of the lowest point mass with the floor. Accordingly, the Poincar\'e map $T$ describes the dynamics 
from one collision to the next. It preserves the smooth measure $\mu$, obtained from the symplectic volume form 
on $\mathbb{R}^N \times \mathbb{R}^N$ via symplectic reduction. Out of historic convenience we will refer 
to the falling point masses as falling balls. 

An intrinsic obstacle, which makes the treatment of this system challenging, is the presence of singular 
collisions. In physical terms, they occur in a triple collision or when the two lower balls hit the floor 
simultaneously. The singular collisions or singularities form codimension one submanifolds in phase space. 
The Poincar\'e map is not well defined on the singularities because it has two images. 

The main question in Wojtkowski's original paper \cite{W90a} revolved around the existence 
of non-zero Lyapunov exponents. Sim\'anyi settled the general case by proving that an arbitrary 
number of falling balls have non-zero Lyapunov exponents almost everywhere \cite{S96}. For a family 
of potential fields $V(q)$, satisfying $\partial V(q) / \partial q > 0$, 
$\partial^2 V(q) / \partial q^2 < 0$, Wojtkowski proved the same result in \cite{W90b}. 
The latter family of potentials does not include the linear potential field.

A new treatment, which ties in old and new ideas, can be found in Wojtkowski's latest work on falling balls \cite{W98}. He starts off with $N$, $N \geq 2$, horizontally aligned balls falling to a moving floor, establishes complete hyperbolicity, and then carries the result over to a variety of falling ball systems by applying stacking rules on the balls. In the most extreme 
case he obtains his original system from \cite{W90a}. The billiard system of each falling ball system 
corresponds to a particle falling in a particular wedge. The form of the wedge depends thereby on the masses and 
the physical model. 

The main line of this work concerns the long time open conjecture whether three (or more) falling balls are ergodic. 
There are two results, 
confirming the ergodicity of two falling balls with mass configurations $m_1 > m_2$: One for the linear 
potential mentioned above \cite{LW92} and one \cite{Ch91} for the family of potentials considered above with the relaxed 
assumption $\partial^2 V(q) / \partial q^2 \leq 0$ and the additional restrictions 
$0 < C_1 \leq \partial V(q) / \partial q \leq C_2 < \infty$, 
$0 \leq \left| \partial^2 V(q) / \partial q^2 \right| \leq C_3 < \infty$, for some constants 
$C_1, C_2, C_3 > 0$. 

Since our system satisfies the mild conditions of Katok-Strelcyn \cite{KS86}, 
the theory of the latter implies that the phase space decomposes into at most countably many ergodic 
components.\footnote{An ergodic component is a set of positive measure in phase space on which the conditional 
smooth measure is ergodic.} From here, it is common to verify the Local Ergodic Theorem 
(LET) together with a transitivity argument to prove the existence of only one ergodic component of full measure 
and, thus, the ergodicity of the system. The LET dates back to Sinai's seminal proof of ergodicity for two 
discs moving uniformly in the two dimensional torus \cite{S70} and was later generalized in the framework of 
semi-dispersing billiards \cite{ChS87,KSSz90,BChSzT02}. In order to prove ergodicity we will use the LET, 
formulated for symplectic maps by Liverani and Wojtkowski \cite{LW92}. 

The LET claims, that one can find an open neighbourhood of a hyperbolic point $p$ with sufficient expansion, which lies 
(mod 0) in one ergodic component, if the following five conditions are satisfied: 
\begin{enumerate}
\item[\ref{C1}] Regularity of singularity manifolds.
\item[\ref{C2}] Non-contraction property.
\item[\ref{C3}] Continuity of transversal Lagrangian subspaces.
\item[\ref{C4}] Chernov-Sinai ansatz.
\item[\ref{C5}] Proper alignment.
\end{enumerate}
Assuming the validity of the LET and the abundance of sufficiently expanding points,\footnote{The abundance 
of sufficiently expanding points is equivalent to saying that the set of sufficiently expanding points 
has measure one and is arcwise connected (cf. Subsection \ref{sec:ab}).} ensures that the neighbourhoods of the LET can be connected to one ergodic component of full measure. 

The bulk of effort in this paper consists in giving a conditional proof of the non-contraction property \ref{C2}. We will use 
coordinate transformations (\ref{hv}), (\ref{xieta}) for which the derivative of the flow equals the identity matrix. 
Hence, it is 
equivalent to verify the non-contraction property for the flow. The paramount advantage is that it is easier for 
us to express results in finite times rather than arbitrarily many derivative map compositions (cf. Section \ref{sec:noncon}).
There are two main ingredients for the proof of the non-contraction property: The first one requires that along every orbit and for every ball to ball collision there exists a subsequence of collision times, such that the pre-collisional velocity differences of the ball to ball collisions are uniformly bounded from below 
(cf. Theorem \ref{heart}). The latter will imply that in every finite interval $[0,T]$, $T > 0$, the number of 
ball to ball collisions is uniformly bounded from above by a constant which depends only on the length $T$ and the 
energy $c > 0$ of the system (cf. Lemma \ref{collisionbound}). The second ingredient requires the validity of the 
strict unboundedness property for every phase point (cf. Assumption \nameref{strict}, Section \ref{2}), i.e. the 
divergence of the 
quadratic form $Q$ along every orbit and every vector of the closed expanding cone field (cf. Definition \ref{unboundedness}). 
For a constant $E_0 > 0$, 
the strict unboundedness property will help us to determine a time $\mathtt{T} = \mathtt{T}(E_0) > 0$, for which 
the $Q$-value of every vector from the expanding cone field has a uniform lower bound $E_0$ (cf. Lemma \ref{taubound2}). This allows us to split the proof 
of the non-contraction property into two parts: First, we prove that the non-contraction property holds for every $t \leq 
\mathtt{T}$ and, second, for every $t > \mathtt{T}$. Note, that the uniform lower bound of the velocity differences and its 
implications is used for the first part only. 

Conditioning the validity of the non-contraction 
property to the validity of strict unboundedness has the advantage that we free ourselves from having to find a Lyapunov 
semi-norm for this model, which is an inherently difficult task by itself. Additionally, as we discuss further below, the 
strict unboundedness property has already been verified for three falling balls with mass configuration given in 
(\ref{special}) \cite{HT20}.\footnote{In fact, the bold reader may verify that the result of Theorem \ref{heart} can be 
implemented into \cite{HT20}, which will yield strict unboundedness for every mass configuration $m_1 > m_2 > m_3$.} 

It is already known, that the continuity of Lagrangian subspaces \ref{C3} is true for 
an arbitrary number of balls \cite{W90a,W91}. 
For the special restriction of masses (\ref{special}) the configuration space (\ref{configurationspace}) can be unfolded to a 
wide wedge \cite[Definition 6.1]{W98}. Wojtkowski's insight \cite{W16} allowed to verify condition \ref{C5}, by showing that, due to the unfolding of the wedge, orbits hitting the unaligned triple collision singularity manifold can be uniquely continued 
\cite[Subsection 7.3]{HT20}. Except for the missing triple singularity manifold, this system is identical to the system of 
falling balls up to a $Q$-isometric coordinate transformation (\ref{fbtowedge}). In order to distinguish between the two systems we follow Wojtkowski \cite{W98} and call the former system a particle falling in a wide wedge system.

Since these systems relate to each other via a $Q$-isometric coordinate transformation it suffices to verify the 
conditions of the LET and the abundance of sufficiently expanding points in only one of the systems. 

For the particular mass restrictions (\ref{special}) we proved the strict unboundedness property, the Chernov-Sinai ansatz 
\ref{C4} and the abundance of sufficiently expanding points \cite{HT20}. Using \cite[Lemma 7.7]{LW92}, it takes not much 
effort to check the regularity of singularity manifolds \ref{C1} (cf. Section \ref{sec:wedge}). Since the strict unboundedness 
assumption is valid, the new result of this work gives that the non-contraction property \ref{C2} is valid as well. Therefore, 
we arrive at the conclusion that a particle falling in a three dimensional wide wedge is ergodic.

\section{Main results}\label{2}
The phase space $\mathcal{M}$ is partitioned ($\operatorname{mod} 0$) into subsets $\mathcal{M}_i$, 
$i = 1, \ldots, N$. $\mathcal{M}_1$ contains the states right after a collision with the floor and 
$\mathcal{M}_i$, $i = 2, \ldots, N$, contains the states right after a collision of balls $i-1$, $i$. 
The Poincar\'e map $T:\ \mathcal{M} \circlearrowleft$ describes the movement from one collision to 
the next. After applying Wojtkowski's convenient coordinate transformation $(q,p) \to (h,v) \to (\xi, \eta)$ 
(\ref{hv}), (\ref{xieta}), we obtain an expanding cone field 
$\{\mathcal{C}(x):\ x \in \mathcal{M}\}$, explicitly given by
\begin{align}
\mathcal{C}(x) &= \{(\delta \xi, \delta \eta) \in \mathbb{R}^{N} \times \mathbb{R}^{N}:\ 
Q(\delta \xi, \delta \eta) > 0,\ \delta \xi_0 = 0,\ \delta \eta_0 = 0\} \cup \{\vec{0}\},\nonumber\\
\mathcal{C}^{\prime}(x) &= \{(\delta \xi, \delta \eta) \in \mathbb{R}^{N} \times \mathbb{R}^{N}:\ 
Q(\delta \xi, \delta \eta) < 0,\ \delta \xi_0 = 0,\ \delta \eta_0 = 0\} \cup \{\vec{0}\}.\nonumber 
\end{align}
where $(\delta \xi, \delta \eta) = (\delta \xi_0, \ldots, \delta \xi_{N-1}, \delta \eta_0, \ldots, 
\delta \eta_{N-1})$ denote the coordinates in tangent space. The quadratic form $Q$ is defined 
(cf. Definition \ref{Qdef}) by a pair of constant, transversal Lagrangian subspaces (\ref{constant}) and the 
symplectic form $\omega$. For this choice of Lagrangian subspaces $Q$ becomes the Euclidean inner product
\begin{align}
Q(\delta \xi, \delta \eta) = \langle \delta \xi, \delta \eta \rangle = 
\sum_{i = 1}^{N-1} \delta \xi_i \delta \eta_i.\nonumber
\end{align}
Denote by $\overline{\mathcal{C}(x)}$ the closure of the cone $\mathcal{C}(x)$, let 
$d_xT^n = d_{T^nx}T \ldots d_{Tx}T d_xT$ and 
$(d_{T^n x}T)_{n \in \mathbb{N}} = (d_xT, d_{Tx}T, d_{T^2x}T,\ldots)$. 
The sequence $(d_{T^n x}T)_{n \in \mathbb{N}}$ is called unbounded, if
\begin{align}
\lim_{n \to +\infty} Q(d_{x}T^{n}v) = +\infty,\ \forall\ v \in \mathcal{C}(x) \setminus \{\vec{0}\},\nonumber
\end{align}
and strictly unbounded, if
\begin{align}
\lim_{n \to +\infty} Q(d_{x}T^{n}v) = +\infty,\ \forall\ v \in \overline{\mathcal{C}(x)} \setminus 
\{\vec{0}\}.\nonumber
\end{align}
For the proof of the non-contraction property (cf. Theorem A below), we have to assume that 
$(d_{T^n x}T)_{n \in \mathbb{N}}$ is strictly unbounded for every $x \in \mathcal{M}$.
\begin{ass}[SU]\label{strict}
  For every $x \in \mathcal{M}$, we have 
	\begin{align}
	\lim_{n \to + \infty}Q(d_xT^n(\delta \xi, \delta \eta)) = +\infty,\nonumber
	\end{align}
	for all $(\delta \xi, \delta \eta) \in \overline{\mathcal{C}(x)} \setminus \{\vec{0}\}$.
\end{ass}
Due to Proposition 6.2 and Theorem 6.8 of \cite{LW92}, Assumption (\nameref{strict}) also implies the strict 
unboundedness for the orbit in negative time $(d_{T^n x}T)_{n \in \mathbb{Z}^-}$, i.e.
\begin{align}
\lim_{n \to -\infty} Q(d_xT^n v) = -\infty,\ \forall\ v \in \overline{\mathcal{C}^{\prime}(x)} \setminus 
\{\vec{0}\}.\nonumber
\end{align} 
The singularity manifold on which $T$ resp. $T^{-1}$ is not well-defined is given by $\mathcal{S}^+$ 
resp. $\mathcal{S}^-$. Let $\mu_{\mathcal{S}^+}$ resp. $\mu_{\mathcal{S}^-}$ be the 
measures induced on the codimension one hypersurfaces $\mathcal{S}^+$ resp. $\mathcal{S}^-$, 
from the smooth $T$-invariant measure $\mu$. We further abbreviate
\begin{align}
\mathcal{S}_n^{\pm} = \mathcal{S}^{\pm} \cup T^{\mp 1}\mathcal{S}^{\pm} \cup \ldots \cup T^{\mp (n-1)}\mathcal{S}^{\pm}.\nonumber
\end{align}
Under assumption (\nameref{strict}), we prove the non-contraction property which is one of the five 
conditions \ref{C1} - \ref{C5} of the LET (cf. Section \ref{sec:ergodicity})
\begin{thma}[Non-contraction property]
Assume that assumption (\nameref{strict}) holds. Then, there exists $\zeta > 0$, such that
\begin{enumerate}[resume]
\item for every $n \geq 1$, $x \in \mathcal{M} \setminus \mathcal{S}_n^+$, and
$(\delta \xi, \delta \eta) \in \overline{\mathcal{C}(x)}$, we have
\begin{align}
\|d_xT^n(\delta \xi, \delta \eta)\| \geq \zeta \|(\delta \xi, \delta \eta)\|,\nonumber
\end{align}
\item for every $n \geq 1$, $x \in \mathcal{M} \setminus \mathcal{S}_n^-$, and
$(\delta \xi, \delta \eta) \in \overline{\mathcal{C}^{\prime}(x)}$, we have
\begin{align}
\|d_xT^{-n}(\delta \xi, \delta \eta)\| \geq \zeta \|(\delta \xi, \delta \eta)\|.\nonumber
\end{align}
\end{enumerate}
\end{thma}
For three falling balls with the additional mass restriction (\ref{special}), 
the configuration space can be unfolded to a billiard table where the ominous proper alignment condition 
\ref{C5} is satisfied \cite{W16,HT20}. The reason is, that the unfolded system misses the unaligned 
triple collision singularity manifold, since every orbit passing through it can be smoothly continued. 
Except for the missing triple singularity manifold, the system is identical to the system of falling balls up 
to a $Q$-isometric coordinate transformation (\ref{fbtowedge}). In order to distinguish between the two systems we 
follow Wojtkowski \cite{W98} and call such a system a particle falling in a wide wedge. Assumption (\nameref{strict}) 
was proven for three falling balls with mass configurations (\ref{special}) in \cite{HT20}. Therefore, according to 
Theorem A, the non-contraction property holds for this system. Incorporating complementary 
previous results from \cite{HT20} we will prove in Section \ref{sec:wedge} 
\begin{thmb}[Ergodicity]
Consider the system of 3 falling balls with mass restrictions (\ref{special}). Then, the unfolded system of a particle falling in a three dimensional wide wedge is ergodic.
\end{thmb}

\section{The system of \textit{N} falling balls}\label{sec:fb}
Let $q_{i} = q_{i}(t)$ be the position, $p_{i} = p_{i}(t)$ the momentum and $v_i = v_i(t)$ the velocity
of the $i$-th ball. The balls are aligned on top of each other and are therefore confined to
\begin{align}\label{configurationspace}
\mathbf{N}_q = \{(q,p) \in \mathbb{R}^N \times \mathbb{R}^N:\ 0 \leq q_1 \leq \ldots
\leq q_N\}
\end{align}
where the subindex $q$ in $\mathbf{N}_q$ refers to the coordinates $(q,p)$.
The momenta and the velocities are related by $p_i = m_iv_i$. We
assume that the masses $m_i$ decrease strictly as we go upwards $m_1 > \ldots > m_N$.
The movement of the balls are the
result of a linear potential field and their kinetic energies. The total energy of the system is given
by the Hamiltonian function
\begin{align}
H(q,p) = \sum_{i = 1}^{N} \frac{p_i^2}{2m_i} + m_iq_i.\nonumber
\end{align}
The Hamiltonian equations are
\begin{align}\label{equations}
\begin{array}{ccc}
\dot{q_i} & = & \dfrac{p_i}{m_i},\\[0.3cm]
\dot{p_i} & = & -m_i.
\end{array}
\end{align}

The dots indicate differentiation with respect to time $t$ and the Hamiltonian vector field on the
right hand side will be denoted as $X_H = X_H(q,p)$.
For some energy value $c > 0$, the energy manifold $E_c$ and its tangent space $\mathcal{T}E_c$ are given by
\begin{align}\label{tangentenergy}
E_c &= \{(q,p) \in \mathbb{R}_{+}^{N} \times \mathbb{R}^{N}:\ H(q,p) =  \sum_{i = 1}^N
\frac{p_i^2}{2m_i} + m_iq_i = c\},\nonumber\\
\mathcal{T}_{(q,p)}E_c &= \{(\delta q, \delta p) \in \mathbb{R}^N \times \mathbb{R}^N:\
\nabla_{(q, p)}H(\delta q, \delta p)
= \sum_{i=1}^N\dfrac{p_i\delta p_i}{m_i} + m_i\delta q_i = 0\}.
\end{align}
Including the restriction of the positions amounts to $E_c \cap \mathbf{N}_q$.
The Hamiltonian vector field (\ref{equations}) gives rise to the Hamiltonian flow
\begin{align}
 \phi:\ &\mathbb{R} \times \left(E_c \cap \mathbf{N}_q\right) \to E_c \cap \mathbf{N}_q, \nonumber\\
 &(t,(q,p)) \mapsto \phi(t,(q,p)).\nonumber
\end{align}
For convenience, the image will also be written with the time variable as superscript,
i.e. $\phi(t,(q,p)) = \phi^{t}(q,p)$.\\
The standard symplectic form $\omega=\sum_{i = 1}^{N} dq_i \wedge dp_i$ induces the symplectic volume
element $\Omega = \bigwedge_{i = 1}^N \omega$.
The volume element on the energy surface is obtained by contracting $\Omega$,
by a vector $u$, where $u$ is a
vector satisfying $dH(u) = 1$. Denoting the contraction operator by $\iota$, the
volume element on the energy surface is given by $\iota (u) \Omega$. Since the flow preserves
the standard symplectic form, it preserves the volume element and, hence, the Liouville measure $\nu$
on $E_c \cap \mathbf{N}_q$ obtained from it.

We define the Poincar\'e section, which describes the states right after a collision as
$\mathcal{M} = \mathcal{M}_1 \cup \ldots \cup \mathcal{M}_N$, with
\begin{align}
&\mathcal{M}_1 := \{(q,p) \in E_c \cap \mathbf{N}_q: q_1 = 0,\ p_1/m_1 \geq 0\},
\nonumber\\
&\mathcal{M}_l := \{(q,p) \in E_c \cap \mathbf{N}_q: q_{l-1} = q_l,\
p_{l-1}/m_{l-1} \leq p_l/m_l\},\ l = 2, \ldots, N.\nonumber
\end{align}
The set of states right before collision $\mathcal{M}^b = \mathcal{M}_1^b \cup \ldots
\cup \mathcal{M}_N^b$, are defined by
\begin{align}
&\mathcal{M}_1^b := \{(q,p) \in E_c \cap \mathbf{N}_q: q_1 = 0,\ p_1/m_1 < 0\},\nonumber\\
&\mathcal{M}_l^b := \{(q,p) \in E_c \cap \mathbf{N}_q: q_{l-1} = q_l,\
p_{l-1}/m_{l-1} > p_l/m_l\},\ l = 2, \dots, N.\nonumber
\end{align}
The collision between balls $i$ and $i+1$ is fully elastic, i.e. the total momentum and
the kinetic energy are preserved. Therefore, the momenta resp. velocities change according to
\begin{align}\label{collisionqp}
\begin{array}{ccc}
p_{i}^+ & = & \gamma_i p_{i}^- + (1 + \gamma_i)p_{i+1}^-,\\[0.1cm]
p_{i+1}^+ & = & (1 - \gamma_i)p_{i}^- - \gamma_i p_{i+1}^-,\\[0.2cm]
v_{i}^+ & = & \gamma_i v_{i}^- + (1 - \gamma_i)v_{i+1}^-,\\[0.1cm]
v_{i+1}^+ & = & (1 + \gamma_i)v_{i}^- - \gamma_i v_{i+1}^-,
\end{array}
\end{align}
where $\gamma_{i} = (m_{i} - m_{i+1})/(m_{i} + m_{i+1})$, $i = 1, \dots, N-1$. When the bottom particle
collides with the floor the sign of its momentum resp. velocity is reversed
\begin{align}\label{v1p1}
\begin{array}{ccc}
p_1^+ = -p_1^-,\\[0.1cm]
v_1^+ = -v_1^-.
\end{array}
\end{align}

This is derived from (\ref{collisionqp}), by setting the floor velocity $v_0$ zero and letting the floor 
mass $m_0$ go to infinity. As a result, the floor collision does not preserve the total momentum.

These collision laws are described by the linear, symplectic, involutory collision map
\begin{align}
\begin{array}{rl}
 \Phi_{i-1,i}:\ &\mathcal{M}^b \to \mathcal{M},\nonumber\\[0.2cm]
 &(q,p^-) \mapsto (q,p^+).\nonumber
\end{array}
\end{align}

We will write $\Phi$ if we do not want to refer to any specific collision.
Let
\begin{align}\label{returntime}
\tau: E_c \cap \mathbf{N}_q \to \mathbb{R}_+,
\end{align}
be the first return time to $\mathcal{M}^b$. We define the Poincar\'e map as
\begin{align}
T:\  &\mathcal{M} \to \mathcal{M},\nonumber \\
&(q,p) \mapsto \Phi \circ \phi^{\tau(q,p)}(q,p).\nonumber
\end{align}
$T$ is the collision map, that maps the state from right after one collision to the next.

On $\mathcal{M}$, we obtain the volume element $\iota(X_H) \iota (u) \Omega$, by contracting the volume element
$\iota (u) \Omega$ on the energy surface with respect to the direction of the flow $X_H$.
This exterior form defines a smooth measure $\mu$ on $\mathcal{M}$, which is $T$-invariant.

Matching the present state with the next collision in the future resp. the past,
we obtain two (mod 0) partitions of $\mathcal{M}$
\begin{align}
\mathcal{M} = \mathcal{M}_{1,1}^+ \cup \bigcup_{i = 1}^N
\bigcup_{\substack{j = 1 \\ j \neq i}}^N \mathcal{M}_{i,j}^+ =
\mathcal{M}_{1,1}^- \cup \bigcup_{i = 1}^N
\bigcup_{\substack{j = 1 \\ j \neq i}}^N \mathcal{M}_{i,j}^-,\nonumber
\end{align}
where
\begin{align}
&\mathcal{M}_{1,1}^+ = \{x \in \mathcal{M}_1:\ Tx \in \mathcal{M}_1\},\nonumber\\
&\mathcal{M}_{i,j}^+ = \{x \in \mathcal{M}_i:\ Tx \in \mathcal{M}_j\},\
1 \leq i,j \leq N,\ j \neq i,\nonumber\\
&\mathcal{M}_{1,1}^- = \{x \in \mathcal{M}_1:\ T^{-1}x \in \mathcal{M}_1\},\nonumber\\
&\mathcal{M}_{i,j}^- = \{x \in \mathcal{M}_i:\ T^{-1}x \in \mathcal{M}_j\},\
1 \leq i,j \leq N,\ j \neq i.\nonumber
\end{align}
For some instances, it is useful to define the subset 
\begin{align}
\mathcal{M}_{1,1}^{m,+} := \mathcal{M}_{1,1}^+ \cap T^{-1}\mathcal{M}_{1,1}^{+} \cap \ldots \cap 
T^{-m}\mathcal{M}_{1,1}^{+} \subset \mathcal{M}_{1,1}^+,\ m \geq 1,\nonumber
\end{align} 
which contains the states returning $(m+1)$-times to the floor. 

Each partition element $\mathcal{M}_{i,j}^{\pm}$ has a boundary $\partial \mathcal{M}_{i,j}^{\pm}$ and the
intersection of two elements of the same partition is strictly contained in the intersection of their
boundaries, i.e.
\begin{align}
\mathcal{M}_{i,j}^{\pm} \cap \mathcal{M}_{k,l}^{\pm} \subset \partial \mathcal{M}_{i,j}^{\pm} \cap
\partial \mathcal{M}_{k,l}^{\pm},\ (i,j) \neq (k,l).\nonumber
\end{align}

The boundary of each partition consists of a regular part $\mathcal{R}^{\pm}$ and a singular part
$\mathcal{S}^{\pm}$, where we set $\partial \mathcal{M}^{\pm} = \mathcal{R}^{\pm} \cup \mathcal{S}^{\pm}$.
The singular part comprises the following codimension one submanifolds
\begin{align}
&\mathcal{S}_{j,i}^+ = \mathcal{M}_{j,i}^+ \cap \mathcal{M}_{j,i+1}^+,\
\mathcal{S}_{i,j}^- = \mathcal{M}_{i,j}^- \cap \mathcal{M}_{i+1,j}^-,
\nonumber\\[0.1cm]
&i = 2, \ldots, N-1,\ j = 1, \ldots, N,\ j \neq i, i+1,\nonumber\\[0.2cm]
&\mathcal{S}_{k,1}^+ = \mathcal{M}_{k,1}^+ \cap \mathcal{M}_{k,2}^+,\
\mathcal{S}_{1,k}^- = \mathcal{M}_{1,k}^- \cap \mathcal{M}_{2,k}^-,\nonumber\\[0.1cm]
&k = 1, \ldots , N,\  k \neq 2.\nonumber
\end{align}
These sets are called singularity manifolds. The states in $\mathcal{S}_{j,i}^{\pm}$ face
a triple collision next, while the states in $\mathcal{S}_{k,1}^+$, $\mathcal{S}_{1,k}^-$
experience a collision of the lower two balls with the floor next. The maps $T$ resp.
$T^{-1}$ have two different images and are therefore not
well-defined on the sets $\mathcal{S}_{j,i}^{+}$, $\mathcal{S}_{k,1}^{+}$ resp.
$\mathcal{S}_{i,j}^{-}$, $\mathcal{S}_{1,k}^{-}$, because the compositions
$\Phi_{i-1,i} \circ \Phi_{i,i+1}$ and $\Phi_{0,1} \circ \Phi_{1,2}$ do not commute.
In this case, we follow the convention, that the orbit branches into two suborbits and we
continue the system on each branch separately.
We abbreviate, for $n \geq 1$,
\begin{align}
\mathcal{S}^{\pm} = \bigcup_{i = 2}^{N-1}
\bigcup_{\substack{j = 1 \\ j \neq i, i+1}}^N \mathcal{S}_{i,j}^{\pm} \cup
\bigcup_{\substack{k = 1 \\ k \neq 2}}^N \mathcal{S}_{k,1}^+ \cup 
\bigcup_{\substack{k = 1 \\ k \neq 2}}^N \mathcal{S}_{1,k}^-,\
\mathcal{S}_n^{\pm} = \mathcal{S}^{\pm} \cup T^{\mp 1}\mathcal{S}^{\pm} \cup \ldots \cup
T^{\mp (n-1)}\mathcal{S}^{\pm}.\nonumber
\end{align}
In the upcoming sections we also want to refer to the singularity manifolds for the 
flow. We define them, informally, as $\mathcal{S}_t^{\pm}$, for every $t \in \mathbb{R}_+$.

Similarly to $\mathcal{S}^{\pm}$, the $T^{\pm 1}$-image of all points in
$\mathcal{R}^{\pm}$ consists of two simultaneous collisions. The key difference to
singular points is that the derivatives of the involved collision maps commute. This
follows from the fact, that the two pairs of collisions do not share a common ball.
Hence, for regular points our orbit does not split into two suborbits and can therefore
be continued uniquely.
Since the collision maps of the simultaneous collisions for points in
$\mathcal{R}^{\pm}$ commute and $T$ is well-defined on
$\mathcal{S}^- \setminus \mathcal{S}^+$, the regularity properties of the flow and the
collision map imply that, for $n \geq 1$,
\begin{align}\label{symplectomorphism}
T^n:\ \mathcal{M} \setminus \mathcal{S}_n^+ \to \mathcal{M} \setminus \mathcal{S}_n^-
\end{align}
is a symplectomorphism, i.e. $T$ extends diffeomorphically to $\mathcal{R}^+$.

\section{Lyapunov exponents}\label{Lyapunov}
The study of Lyapunov exponents was carried out using a method developed by Wojtkowski
in the string of papers \cite{W85,W88,W91,LW92,W00}. This method has been successfully
implemented to derive that an arbitrary number of falling balls has non-zero Lyapunov
exponents almost everywhere \cite{S96}. The basic tools of the Lyapunov
exponent machinery were further advanced and are inevitable in the study of
ergodicity of Hamiltonian systems \cite{LW92}.
We are therefore going to formulate the fundamentals of this method and how it applies
to the system of falling balls.\\
The standard symplectic form $\omega=\sum_{i = 1}^N dq_i\wedge dp_i$ is given by
\begin{align}
\omega (v_1, v_2) = \left\langle v_{1,1} , v_{2,2} \right\rangle - \left\langle v_{2,1}, v_{1,2} \right\rangle,
\nonumber
\end{align}
where $v_i = (v_{i,1}, v_{i,2}) \in \mathbb{R}^N \times \mathbb{R}^N$, $i = 1,2$.
A Lagrangian subspace $V$ is a subspace of dimension $N$ which is the $\omega$-orthogonal
complement to itself, i.e. the symplectic form is zero for every input from $V$
\cite[Definition 6.4]{LM87}. It is equivalently the subspace of maximal dimension on
which $\omega$ vanishes. Note, that for two transversal Lagrangian subspaces $V_1$,
$V_2$, every vector $v \in \mathbb{R}^N \times \mathbb{R}^N$ has a unique decomposition
$v = v_1 + v_2$, $v_1 \in V_1$, $v_2 \in V_2$.
\begin{definition}
For two transversal Lagrangian subspaces $V_1$, $V_2$ we define the cone between $V_1$ and $V_2$ by
\begin{align}
\mathcal{C}_{V_1,V_2} = \{v \in \mathbb{R}^N \times \mathbb{R}^N:\ \omega (v_1, v_2) > 0,\ v = v_1 + v_2,\
v_i \in V_i,\ i=1,2\} \cup \{\vec{0}\}.\nonumber
\end{align}
\end{definition}
\begin{definition}\label{Qdef}
The quadratic form $Q_{V_1,V_2}$, or $Q_{V_1,V_2}$-form, associated to a pair of
transversal Lagrangian subspaces $V_1$, $V_2$ is given by
\begin{align}
Q_{V_1,V_2}:\ &\mathbb{R}^N \times \mathbb{R}^N \to \mathbb{R},\nonumber\\
&v \mapsto \omega(v_1,v_2),\nonumber
\end{align}
where $v = v_1 + v_2$, $v_i \in V_i$, $i = 1, 2$.
\end{definition}
Observe, that the quadratic $Q_{V_1,V_2}$-form is indefinite with signature $(N,N)$
on $\mathbb{R}^N \times \mathbb{R}^N$. With the definitions above, the
quadratic form can be used to define the cone
\begin{align}
\mathcal{C}_{V_1,V_2} = \{v \in \mathbb{R}^N \times \mathbb{R}^N:\ Q_{V_1,V_2}(v) > 0\}
\cup \{\vec{0}\}. \nonumber
\end{align}
The complementary cone of $\mathcal{C}_{V_1,V_2}$ is given by
\begin{align}
\mathcal{C}_{V_1,V_2}^{\prime} = \{v \in \mathbb{R}^N \times \mathbb{R}^N:\
Q_{V_1,V_2}(v) < 0\} \cup \{\vec{0}\}. \nonumber
\end{align}
The arguably simplest expression of $Q_{V_1,V_2}$ can be obtained
by associating it to the standard Lagrangian subspaces given by
\begin{align}\label{standard}
L_1 = \mathbb{R}^N \times \{\vec{0}\},\quad L_2 = \{\vec{0}\} \times \mathbb{R}^N.
\end{align}
For this choice of transversal Lagrangian subspaces we will abbreviate $Q = Q_{L_1,L_2}$
and $\mathcal{C} = \mathcal{C}_{L_1,L_2}$. Further, for $v = v_1 + v_2$,
the $Q$-form reads
\begin{align}
Q(v) = \left\langle v_1, v_2\right\rangle.\nonumber
\end{align}
In \cite{W90a}, Wojtkowski introduced two coordinate transformations, $i = 1, \ldots, N$,
\begin{align}\label{hv}
\begin{aligned}
h_i &= \dfrac{p_i^2}{2m_i} + m_iq_i,\quad v_i &= \dfrac{p_i}{m_i},
\end{aligned}
\end{align}
and
\begin{align}\label{xieta}
\begin{aligned}
  (\xi_0, \xi_1, \dots ,\xi_{N-1})^T &= A^{-1}(h_1,h_2,\dots,h_N)^T \\
  (\eta_0, \eta_1, \dots ,\eta_{N-1})^T &= A^{T}(v_1,v_2,\dots,v_N)^T,
\end{aligned}
\end{align}
where $A$ is an invertible matrix depending only on the masses \cite[p. 520]{W90a}. 
In order to keep calculations concise and lucid, we will work exclusively in $(\xi, \eta)$-coordinates. 

The energy manifold, its tangent space and the Hamiltonian vector field take the form
\begin{align}
E_c &= \{(\xi, \eta) \in \mathbb{R}^{N-1} \times \mathbb{R}^{N-1}:\ H(\xi,\eta) = \xi_0
= c\},\nonumber\\
\mathcal{T}E_c &= \{(\delta \xi, \delta \eta) \in \mathbb{R}^{N-1} \times
\mathbb{R}^{N-1}:\  \nabla_{(\xi, \eta)}H(\delta \xi, \delta \eta) = \delta \xi_0
= 0\},\nonumber\\
X_H(\xi,\eta) &= (0,\ldots,0,-1,0,\ldots,0).\nonumber
\end{align}

Intersecting the standard Lagrangian subspaces (\ref{standard}) in $(\delta \xi, \delta \eta)$-coordinates with the 
tangent space of the energy manifold and quotienting out the flow direction gives
\begin{align}\label{constant}
\begin{aligned}
&\mathbf{L}_1 =  \{(\delta \xi, \delta \eta) \in \mathbb{R}^N \times \mathbb{R}^N:\
\delta \xi_{0} = 0,\ \delta \eta_i = 0,\ i = 0, \ldots, N-1\} \simeq
\mathbb{R}^{N-1},\\[0.2cm]
&\mathbf{L}_2 = \{(\delta \xi, \delta \eta) \in \mathbb{R}^N \times \mathbb{R}^N:\
\delta \eta_{0} = 0,\ \delta \xi_i = 0,\ i = 0, \ldots, N-1\} \simeq \mathbb{R}^{N-1}.
\end{aligned}
\end{align}
Thus, the $Q$-form given by $\mathbf{L}_1$, $\mathbf{L}_2$ reduces
to $\mathbb{R}^{N-1} \times \mathbb{R}^{N-1}$ and now amounts to
\begin{align}\label{Qxieta}
 Q(\delta \xi, \delta \eta) = \langle \delta \xi, \delta \eta \rangle = \sum_{i=1}^{N-1}
\delta \xi_i \delta \eta_i,
\end{align}
with no further restrictions, when inserting a vector from $\mathbf{L}_1 \oplus \mathbf{L}_2$.

In these coordinates, the derivative of the flow $d\phi^t$ equals the identity map.
Thus, only the derivatives of the collision maps $d\Phi_{i,i+1}$ are relevant to the
dynamics in tangent space. Since $\delta \xi_0 = 0$, $\delta \eta_0 = 0$ we can reduce
the derivatives of the collision maps to $(2N-2 \times 2N-2)$-matrices. In these
coordinates they are given by
\begin{align}\label{dxieta}
d\Phi_{0,1} =
\begin{pmatrix}
\operatorname{id}_{N-1} & 0\\
B & \operatorname{id}_{N-1}
\end{pmatrix},\
d\Phi_{i,i+1} =
 \begin{pmatrix}
D_i & F_i\\
0 & D_i^T
\end{pmatrix},\
i = 1, \ldots, N-1,
\end{align}
where $B = (b_{m,n})_{m,n = 1}^{N-1}$, $F_i = (f_{m,n})_{m,n = 1}^{N-1}$ have the structure of
the zero matrix, except for the entries $b_{1,1} = \beta$, $f_{i,i} = -\alpha_{i}$ and
$D_i = (d_{m,n})_{m,n = 1}^N$ has the structure of the identity matrix, except
for the following entries in the $i$-th row
\begin{align}
  d_{i,i-1} = 1-\gamma_i,\quad d_{i,i} = -1,\quad d_{i,i+1} = 1+\gamma_i.\nonumber
\end{align}
The terms $\alpha_1, \ldots, \alpha_N$ and $\beta$ in the matrices are non-negative and
given by
\begin{align}\label{alpha}
\beta = -\dfrac{2}{m_1v_1^-},\quad \alpha_{i} =
\dfrac{2m_im_{i+1}(m_i - m_{i+1})(v_i^- - v_{i+1}^-)}{(m_i + m_{i+1})^2}.
\end{align}
Observe, that the strict inequality $m_1 > \ldots > m_N$ of the mass configurations implies, that
$\alpha_i > 0$, since $v_i^- - v_{i+1}^- > 0$.

Using the quadratic form $Q$, we define the open cone $\mathcal{C}$ and the
complementary cone $\mathcal{C}^{\prime}$ associated to the Lagrangian subspaces $\mathbf{L}_1$, $\mathbf{L}_2$ by
\begin{align}
\mathcal{C} &= \{(\delta \xi, \delta \eta) \in \mathbf{L}_1 \oplus \mathbf{L}_2:\
Q(\delta \xi, \delta \eta) = \langle \delta \xi, \delta \eta \rangle > 0\} \cup \{\vec{0}\},\nonumber\\
\mathcal{C}^{\prime} &= \{(\delta \xi, \delta \eta) \in \mathbf{L}_1 \oplus
\mathbf{L}_2:\
Q(\delta \xi, \delta \eta) = \langle \delta \xi, \delta \eta \rangle < 0\} \cup \{\vec{0}\}.\nonumber
\end{align}
The cone field $\left\{ \mathcal{C}(x):\ x \in \mathcal{M} \right\}$ is constant and
therefore continuous in $\mathcal{M}$.
Denote by $\overline{\mathcal{C}}$ the closure of the cone $\mathcal{C}$.
\begin{definition}\label{Defq}
1. The cone $\mathcal{C}$ is called \textit{invariant} at $x \in \mathcal{M}$,
if
 \begin{align}
 d_x T \overline{\mathcal{C}} \subseteq \overline{\mathcal{C}},\nonumber
 \end{align}
2. The cone $\mathcal{C}$ is called
\textit{strictly invariant} at $x \in \mathcal{M}$, if
 \begin{align}
 d_x T \overline{\mathcal{C}} \subseteq \mathcal{C},\nonumber
 \end{align}
3. The cone $\mathcal{C}$ is called
\textit{eventually strictly invariant} at $x \in \mathcal{M}$, if there exists a positive integer
$k = k(x) \geq 1$, such that
 \begin{align}
 d_x T^k \overline{\mathcal{C}} \subseteq \mathcal{C},\nonumber
 \end{align}
4. The map $d_x T$ is called $Q$\textit{-monotone}, if
\begin{align}
Q(d_x T v) \geq Q(v), \nonumber
\end{align}
for all $v \in \mathbf{L}_1 \oplus \mathbf{L}_2$.\\
5. The map $d_x T$ is called \textit{strictly $Q$-monotone}, if
\begin{align}
Q(d_x T v) > Q(v), \nonumber
\end{align}
for all $v \in \mathbf{L}_1 \oplus \mathbf{L}_2 \setminus \{\vec{0}\}$.\\
6. The map $d_x T$ is called \textit{eventually strictly $Q$-monotone}, if there exists a
positive integer $k = k(x) \geq 1$, such that
\begin{align}
Q(d_x T^k v) > Q(v), \nonumber
\end{align}
for all $v \in \mathbf{L}_1 \oplus \mathbf{L}_2 \setminus \{\vec{0}\}$.
\end{definition}
In the definition above, statements 1, 2, 3 are equivalent to statements 4, 5, 6
\cite[Section 4]{LW92}.
In order to obtain non-zero Lyapunov exponents we repeat Wojtkowski's criterion \cite{W85},
which links eventual strict $Q$-monotonicity to non-zero Lyapunov exponents
\begin{qcrit}[Theorem 5.1, \cite{W85}]
If $d_x T$ is eventually strictly $Q$-monotone for $\mu$-a.e. x, then all Lyapunov
exponents are non-zero almost everywhere.
\end{qcrit}
The derivative $d_xT$ is $Q$-monotone for every $x \in \mathcal{M}$ and any number of
falling balls \cite{W90a}.
Sim\'anyi established that $N$, $N \geq 2$, falling balls have non-zero Lyapunov
exponents for $\mu$-a.e. $x \in \mathcal{M}$, by verifying the $Q$-criterion \cite{S96}. 

Observe, that the coordinate transformation (\ref{xieta}) is $Q$-isometric, i.e.
\begin{align}
 Q(\delta \xi, \delta \eta) = Q(A^{-1} \delta h, A^{T} \delta v) =
 Q(\delta h, \delta v), \nonumber
\end{align}
which represents a change of basis inside of both Lagrangian subspaces. Therefore, it does not 
make a difference in terms of the $Q$-form's value whether we operate in $(\delta h, \delta v)$ 
or $(\delta \xi, \delta \eta)$-coordinates.

We close this subsection by formulating the (strict) unboundedness property and the least expansion
coefficients, which will be used to establish criteria for ergodicity.\\
The \textit{least expansion coefficients} $\sigma$, $\sigma_{\mathcal{C}^{\prime}}$, for $n \geq 1$, are defined as
\begin{align}
\sigma(d_x T^n) = \inf_{0\ne v \in \mathcal{C}(x)} \sqrt{\frac{Q(d_{x}T^nv)}{Q(v)}}, \quad
\sigma_{\mathcal{C}^{\prime}}(d_x T^{-n}) = \inf_{0\ne v \in \mathcal{C}^{\prime}(x)}
\sqrt{\frac{Q(d_x T^{-n}v)}{Q(v)}}.\nonumber
\nonumber
\end{align}

\begin{definition}\label{unboundedness}
1. The sequence $(d_{T^n x}T)_{n \in \mathbb{N}}$ is called \textit{unbounded}, if
\begin{align}
\lim_{n \to +\infty} Q(d_x T^n v) = +\infty,\ \forall\ v \in \mathcal{C}(x) \setminus \{\vec{0}\}.\nonumber
\end{align}

2. The sequence $(d_{T^n x}T)_{n \in \mathbb{N}}$ is called \textit{strictly unbounded}, if
\begin{align}
\lim_{n \to +\infty} Q(d_x T^n v) = +\infty,\ \forall\ v \in \overline{\mathcal{C}(x)}
\setminus \{\vec{0}\}.\nonumber
\end{align}
\end{definition}
The least expansion coefficient and the property of strict unboundedness relate to each
other in the following way
\begin{theorem}[Theorem 6.8, \cite{LW92}]\label{relation}
The following assertions are equivalent:
\begin{enumerate}
\item The sequence $(d_{T^nx}T)_{n\in\mathbb{N}}$ is strictly unbounded.\\[-0.3cm]
\item $\lim_{n \to \infty} \sigma(d_xT^n) = \infty$.\\[-0.3cm]
\end{enumerate}
\end{theorem}

\begin{remark}\label{csremark}
The strict unboundedness property can also be stated in negative time, i.e.
\begin{align}
\lim_{n \to -\infty} Q(d_xT^n v) = -\infty,\ \forall\ v \in \overline{\mathcal{C}^{\prime}(x)} \setminus
\{\vec{0}\}.\nonumber
\end{align}
Following the proof of \cite[Theorem 6.8]{LW92}, Theorem \ref{relation} also extends to this case, i.e.
\begin{enumerate}[resume]
\item The sequence $(d_{T^nx}T)_{n\in\mathbb{Z}_{-}}$ is strictly unbounded.\\[-0.3cm]
\item $\lim_{n \to \infty} \sigma_{\mathcal{C}^{\prime}}(d_xT^{-n}) = \infty$.\\[-0.3cm]
\end{enumerate}
\end{remark}

\section{Ergodicity}\label{sec:ergodicity}
Due to the theory of Katok-Strelcyn \cite{KS86} we know that our phase space decomposes into at most countably
many components on which the conditional smooth measure is ergodic. The strategy to prove ergodicity involves
two steps:
\begin{enumerate}
\item Proving local ergodicity (or the Local Ergodic Theorem), which implies that every ergodic component is a (mod 0) open set.\\[-0.3cm]
\item Proving that the set of sufficiently expanding points (Definition \ref{sufficiency}) is arcwise connected
and of full measure, which implies that any two (mod 0) open ergodic components can be connected with each
other, such that their intersection is of positive $\mu$-measure.
\end{enumerate}
The validity of both points above proves the existence of only one ergodic component of full measure.

\subsection{Local Ergodicity}
We use the Local Ergodic Theorem (LET) of \cite{LW92} and begin with the definition of a sufficiently expanding
point.
\begin{definition}\label{sufficiency}
A point $p \in \mathcal{M}$ is called \textit{sufficient}
(or \textit{sufficiently expanding}) if there exists a neighbourhood $\mathcal{U} = \mathcal{U}(p)$ and an integer
$\mathtt{N} = \mathtt{N}(p) >0$ such that either
\begin{enumerate}[resume]
\item $\mathcal{U} \cap \mathcal{S}_\mathtt{N}^- = \varnothing$ and $\sigma(d_yT^\mathtt{N}) > 3$,
for all $y \in T^{-\mathtt{N}}\mathcal{U}$, or \label{suff1}
\item $\mathcal{U} \cap \mathcal{S}_\mathtt{N}^+ = \varnothing$ and
$\sigma_{\mathcal{C}^{\prime}}(d_yT^{-\mathtt{N}}) > 3$, for all $y \in T^{\mathtt{N}} \mathcal{U}$.
\label{suff2}
\end{enumerate}
\end{definition}

Note, that in the sufficiency definition the requirements
$\mathcal{U} \cap \mathcal{S}_\mathtt{N}^- = \varnothing$ in (\ref{suff1}) and
$\mathcal{U} \cap \mathcal{S}_\mathtt{N}^+ = \varnothing$ in (\ref{suff2}) additionally
demand, that the orbit meets no singular manifold in the first $\mathtt{N}(p)-1$ iterates.

The LET amounts to showing that around a sufficient point, it is possible to find an
open neighbourhood, which lies (mod 0) in one ergodic component.
\begin{local}\label{local}
Let $p \in \mathcal{M}$ be a sufficient point and let $\mathcal{U} = \mathcal{U}(p)$ be the
neighbourhood from Definition \ref{sufficiency}. Suppose conditions \ref{C1} - \ref{C5} below
are satisfied.
\begin{enumerate}
\item[\mylabel{C1}{(C1)}] $(Regularity\ of\ singularity\ manifolds):$ The sets $\mathcal{S}_n^+$ and $\mathcal{S}_n^-$,
$n \geq 1$, are regular subsets.\footnote{For the definition of a regular subset refer to \cite[Definition 7.1]{LW92}}

\item[\mylabel{C2}{(C2)}] $(Non-contraction\ property):$ There exists $\zeta > 0$, such that
\begin{enumerate}
\item for every $n \geq 1$, $x \in \mathcal{M} \setminus \mathcal{S}_n^+$, and
$(\delta \xi, \delta \eta) \in \overline{\mathcal{C}(x)}$, we have
\begin{align}
\|d_xT^n(\delta \xi, \delta \eta)\| \geq \zeta \|(\delta \xi, \delta \eta)\|,\nonumber
\end{align}
\item for every $n \geq 1$, $x \in \mathcal{M} \setminus \mathcal{S}_n^-$, and
$(\delta \xi, \delta \eta) \in \overline{\mathcal{C}^{\prime}(x)}$, we have
\begin{align}
\|d_xT^{-n}(\delta \xi, \delta \eta)\| \geq \zeta \|(\delta \xi, \delta \eta)\|.\nonumber
\end{align}
\end{enumerate}

\item[\mylabel{C3}{(C3)}] $(Continuity\ of\ Lagrangian\ subspaces):$ The ordered pair of transversal Lagrangian
subspaces $(L_{1}(x)$, $L_{2}(x))$ varies continuously in $\operatorname{int}\mathcal{M}$.

\item[\mylabel{C4}{(C4)}] $(Chernov-Sinai\ ansatz):$ For $\mu_{\mathcal{S}^{\mp}}$-a.e. $x \in \mathcal{S}^{\mp}$,
\begin{align}
\lim_{n \to \pm \infty} Q(d_x T^n v) = \pm \infty, \nonumber
\end{align}
for every $v \in \overline{\mathcal{C}(x)} \setminus \{0\}$, if $x \in \mathcal{S}^-$ and for every
$v \in \overline{\mathcal{C}^{\prime}(x)} \setminus \{0\}$, if $x \in \mathcal{S}^+$.

\item[\mylabel{C5}{(C5)}] $(Proper\ alignment):$ There exists $\mathtt{M} \geq 0$, such that for every
$x \in \mathcal{S}^{+}$ resp. $\mathcal{S}^{-}$, we have $d_xT^{-\mathtt{M}} v_{x}^{+}$ resp.
$d_xT^{\mathtt{M}} v_{x}^{-}$ belong to $\overline{\mathcal{C}^\prime (T^{-\mathtt{M}} x)}$ resp.
$\overline{\mathcal{C}(T^{\mathtt{M}} x)}$, where $v_x^+$ resp. $v_x^-$ are the characteristic
lines\footnote{The characteristic line
$v_{x}^{\pm}$ is a vector of $\mathcal{T}_x\mathcal{S}^{\pm}$ that has the property of
annihilating every other vector $w \in \mathcal{T}_x\mathcal{S}^{\pm}$ with respect to the
symplectic form $\omega$, i.e. $\omega(v_{x}^{\pm},w) = 0$, $\forall\ w \in \mathcal{T}_x\mathcal{S}^
{\pm}$. Alternatively stated, it is the $\omega$-orthogonal complement of $\mathcal{T}_x\mathcal{S}^{
\pm}$. Note, that in symplectic geometry the $\omega$-orthogonal complement of a codimension one
subspace is one dimensional.} of $\mathcal{T}_x\mathcal{S}^{+}$ resp. $\mathcal{T}_x\mathcal{S}^{-}$.
\end{enumerate}
Then, the open neighbourhood  $\mathcal{U}(p)$ is contained (mod 0) in one ergodic component.
\end{local}

\subsection{Abundance of sufficiently expanding points}\label{sec:ab}
The notion of a sufficiently expanding point was given in Definition \ref{sufficiency}.
Once local ergodicity is established we deduce that every ergodic component is (mod 0) open.
One possibility to obtain a single ergodic component is
\begin{theorem}[Abundance of sufficiently expanding points]\label{abundance}
	The set of sufficiently expanding points has full measure and is arcwise connected.
\end{theorem}
The abundance of sufficiently expanding points can be proven at once by requiring the strict unboundedness 
assumption (\nameref{strict}), the proper alignment property \ref{C5}
and the explicit construction of the neighbourhood lying in one ergodic component from the LET in the beginning 
of Section 8 in \cite{LW92}.
\begin{proof}[Proof of Theorem \ref{abundance}]
Recall that a point $x \in \mathcal{M}$ is sufficient if there exists a positive integer
$\mathtt{N} = \mathtt{N}(x) > 0$,
such that either (\ref{suff1}) or (\ref{suff2}) from Definition \ref{sufficiency} are satisfied.
Due to the strict unboundedness (\nameref{strict}), Theorem \ref{relation} and Remark \ref{csremark}, $\sigma(d_xT^n)$ 
and $\sigma_{\mathcal{C}^{\prime}}(d_xT^{-n})$ diverge to infinity for every $x \in \mathcal{M}$. Therefore,
every orbit which experiences at most one singular collision satisfies either
\begin{enumerate}[resume]
\item[(5)] $\sigma(d_x T^{\mathtt{N}(x)}) > 3$, $T^k\mathcal{U} \cap \mathcal{S}^+ = \varnothing$,
$0 \leq k \leq \mathtt{N}(x)$, or\\[-0.3cm]
\item[(6)] $\sigma_{\mathcal{C}^{\prime}}(d_xT^{-\mathtt{N}(x)}) > 3$, $T^{-k}\mathcal{U} \cap \mathcal{S}^- = \varnothing$, $0 \leq k \leq \mathtt{N}(x)$.
\end{enumerate}
It follows that the only non-sufficient orbits lie in a subset of double singular
collisions. Due to the proper alignment property \ref{C5}, $\mathcal{S}^+$ and
$\mathcal{S}^-$ are transversal for
every point, thus, the points of double singular collisions form a set of (at least)
codimension two.
Hence, there is an arcwise connected set of measure one such that the least expansion
coefficient is larger than three. For the last part, the proof follows the beginning
of Section 8 in \cite{LW92}:

Without loss of generality assume that
$\sigma_{\mathcal{C}^{\prime}}(d_xT^{\mathtt{N}}) > 3$. We can choose a small enough neighbourhood
$\mathcal{U}$ around the point $x$ such that $T^{\mathtt{N}}: T^{-\mathtt{N}}\mathcal{U} \to \mathcal{U}$ is a
diffeomorphism. This implies that $\mathcal{U} \cap \mathcal{S}_{\mathtt{N}}^- = \varnothing$ and $
T^{-\mathtt{N}}\mathcal{U} \cap \mathcal{S}_{\mathtt{N}}^+ = \varnothing$.
Further, the functional $y \mapsto \sigma(d_yT^{\mathtt{N}})$ is continuous on $\mathcal{U}$ and by making
$\mathcal{U}$ smaller, if necessary, we obtain $\sigma(d_yT^{\mathtt{N}}) > 3$, for every
$y \in T^{-\mathtt{N}}\mathcal{U}$.
\end{proof}

\section{Uniform lower bound of velocity differences}
The investigation regarding a uniform lower bound of velocity differences
$v_i^- - v_{i+1}^-$, for any $i \in \{ 1, \ldots, N-1\}$, is of main interest for the
non-contraction property. Denote by $(i,i+1)$ the collision between ball $i$ and ball 
$i+1$, i.e. when $q_i = q_{i+1}$.

Let $x = x(t) \in \mathcal{M}_{i+1}$, $i \in \{1, \ldots N-1\}$. The velocity difference
$v_i^-(t) - v_{i+1}^-(t)$ is non-negative and due to the collision laws
(\ref{collisionqp}), changes sign after the collision, i.e.
\begin{align}\label{sign}
 0 \leq v_i^-(t) - v_{i+1}^-(t) = -(v_i^+(t) - v_{i+1}^+(t)).
\end{align}
The Hamiltonian equations imply, that during free flight this quantity remains preserved
(\ref{equations}). Using (\ref{equations}), (\ref{collisionqp}), we see that the term
$v_i^-(t) - v_{i+1}^-(t)$ is only affected by a $(i-1,i)$ resp. $(i+1,i+2)$ collision when
being expanded backwards, i.e.
\begin{align}\label{sandwich}
  \begin{array}{rcl}
 0 \leq v_i^-(t) - v_{i+1}^-(t) &=& (1 + \gamma_{i-1})(v_{i-1}^-(t_c) -
 v_{i}^-(t_c))\\[0.2cm]
 &+& (v_i^-(t_c) - v_{i+1}^-(t_c)),\
 \text{resp.}\\[0.3cm]
 0 \leq v_i^-(t) - v_{i+1}^-(t) &=& (1 - \gamma_{i+1})(v_{i+1}^-
 (t_c) - v_{i+2}^- (t_c))\\[0.2cm]
 &+& (v_i^-(t_c) - v_{i+1}^-(t_c)),
  \end{array}
\end{align}
where $t_c < t$ is the collision time of the $(i-1,i)$ resp. $(i+1,i+2)$ collision.
Since we stopped our expansion right before a $(i-1,i)$ resp.
$(i+1,i+2)$ collision, we have
\begin{align}\label{pre}
  v_{i-1}^-(t_c) - v_{i}^-(t_c) \geq 0,\quad
  v_{i+1}^-(t_c) - v_{i+2}^-(t_c) \geq 0.
\end{align}
Formula (\ref{sandwich}) can be generalized in the following way.
Let $t_1 < t_2$ be collision times of two successful $(i,i+1)$ collisions and
$m, n \in \mathbb{N}$. Assume that in between those two $(i,i+1)$ collisions we
have $m$ $(i-1,i)$ collisions and $n$ $(i+1,i+2)$ collisions, with collision
times $r_1, \ldots, r_m$ and  $u_1, \ldots u_n$. Expanding only the $(i,i+1)$ velocity
difference backwards without changing the appearing $(i-1,i)$ and $(i+1,i+2)$ 
velocity differences, we obtain for $i \geq 2$
\begin{align}\label{sandwichsum}
  \begin{array}{rcl}
  0 &\leq& v_i^-(t_2) - v_{i+1}^-(t_2)\\[0.3cm]
  &=& (1 + \gamma_{i-1})\sum_{j = 1}^{m}
  \bigl(v_{i-1}^-(r_j)- v_{i}^-(r_j)\bigr)\\[0.3cm]
  &+& (1 - \gamma_{i+1})\sum_{l = 1}^{n}
  \bigl(v_{i+1}^-(u_l) - v_{i+2}^- (u_l)\bigr)\\[0.3cm]
  &+& v_i^+(t_1) - v_{i+1}^+(t_1).
  \end{array}
\end{align}
In between two $(1,2)$ collisions, we assume to have one floor collision, 
$m$ full returns to the floor of the lowest ball and $n$ $(2,3)$ collisions, 
again with collision times $r_1, \ldots, r_m$ and  $u_1, \ldots u_n$. 
Expanding $(1,2)$ at $t_2$ backwards yields
\begin{align}\label{sandwichsum12}
  \begin{array}{rcl}
  0 &\leq& v_1^-(t_2) - v_{2}^-(t_2)\\[0.3cm]
  &=& 2\sum_{j = 1}^{m}2jv_1^+(r_j)\\[0.3cm]
  &+& (1 - \gamma_{2})\sum_{l = 1}^{n}
  \bigl(v_{2}^-(u_l) - v_{3}^- (u_l)\bigr)\\[0.3cm]
	&+& 2\sqrt{(v_1^+(t_1))^2 + 2q_1(t_1)} + v_1^+(t_1) - v_{2}^+(t_1).
  \end{array}
\end{align}
If there is at least one floor collision between two $(1,2)$ collisions, then 
the square root term $2\sqrt{(v_1^+(t_1))^2 + 2q_1(t_1)}$ appears in (\ref{sandwichsum12}). 
The latter is part of the time the lowest ball needs to fall to the floor after a $(1,2)$ collision.\footnote{
The exact time the lowest ball needs to fall to the floor is $v_1^+(t_1)+\sqrt{(v_1^+(t_1))^2 + 2q_1(t_1)}$.}
\footnote{Note, that we can have a floor collision between two $(1,2)$ collision without 
a full return of the lowest ball to the floor, i.e. the square 
root term is present in (\ref{sandwichsum12}) but $m = 0$ in the first sum.}

Remember, that
\begin{align}\label{pre01}
v_1^+(r_j) \geq 0,\ \forall\ j \in \{1,\ldots,m\},
\end{align}
since this is the velocity of the first ball right after taking off from the floor.

At the heart of this work lies the following
\begin{theorem}\label{heart}
  There exists a constant $C > 0$, such that for every $x \in \mathcal{M}$ and every
  $i \in \{1, \ldots, N-1\}$, there exists a divergent sequence of collision times
  $(\mathtt{t}_n)_{n \in \mathbb{N}} = (\mathtt{t}_n(x,i))_{n \in \mathbb{N}}:$
  $v_i^-(\mathtt{t}_n) - v_{i+1}^-(\mathtt{t}_n) \geq C$.
\end{theorem}

\textit{Outline of the proof}: We start describing a certain collision pattern.
Since every collision happens infinitely often, this collision pattern can be found 
(non-uniquely) infinitely often in every orbit. 
The time interval of this pattern in the proof below is given by $[t_{-(1,2)}, t_{(1,2)}]$. 
At $t_{-(1,2)}$, $t_{(1,2)}$ we have a $(1,2)$ collision and somewhere in between 
is at least one $(0,1)$ collision. 

We start a proof by contradiction assuming that 
every pre-collisional velocity difference of $(1,2)$ collisions in $[t_{-(1,2)}, t_{(1,2)}]$ 
is arbitrarily small. Using the above formulas (\ref{sandwich}) - (\ref{sandwichsum12}) this will amount to 
having every ball arbitrarily close to the floor with velocities being arbitrarily close 
to each other at time $t_{-(1,2)}$. Since there is at least one $(0,1)$ collision between 
the two $(1,2)$ collisions at $t_{-(1,2)}$, $t_{(1,2)}$  the square root term in (\ref{sandwichsum12}) 
exists. This implies that at $t_{-(1,2)}$ all the balls will have arbitrarily small velocities, 
which results in a contradiction since the energy of the system would be arbitrarily small. 
Hence, the velocity difference $v_1^- - v_2^-$ of at least one $(1,2)$ collision 
in $[t_{-(1,2)}, t_{(1,2)}]$ is bounded from below. 

Repeatedly using the above formulas, 
we obtain lower bounds for at least one velocity difference $v_i^- - v_{i+1}^-$, 
for every $i \in \{1,\ldots,N-1\}$. Since this collision pattern appears infinitely often along 
every orbit we can extend these considerations obtaining the result from Theorem \ref{heart}.

\begin{proof}
Pick an arbitrary $(1,2)$ collision and mark the time as $t_{-(1,2)}$. Then, pick
the next $(2,3)$ collision in the future and mark the time as $t_{-(2,3)}$.
Continuing this procedure for the next $(3,4), \ldots (N-1,N)$ collisions, gives us
collision times $t_{-(3,4)}, \ldots, t_{-(N-1,N)}$. After that we pick the first 
$(0,1)$ collision and mark its collision time with $t_0$. 
We now reverse the order of collisions after
$t_0$ and mark the future collision times of the first consecutively appearing
$(N-1,N), \ldots, (1,2)$ collisions as $t_{(N-1,N)}, \ldots, t_{(1,2)}$. Note, that in the
intervals $[t_{-(i,i+1)},t_{-(i+1,i+2)}]$, $i \in \{1,\ldots,N-2\}$, 
exactly one $(i+1,i+2)$ collision occurs, while in the interval
$[t_{0},t_{(N-1,N)}]$ resp. $[t_{(i,i+1)},t_{(i-1,i)}]$, $i \in \{2,\ldots,N-1\}$, 
exactly one $(N-1,N)$ resp. $(i-1,i)$ collision occurs, 
but there is no restriction on other collisions happening.

The collision times of each $(i,i+1)$ collision, including floor collisions, 
induce a partition $\mathcal{P}_i$ of the time interval $[t_{-(1,2)},t_{(1,2)}]$: 
For every $i \in \{2, \ldots N-1\}$, there exists a positive
integer $n = n(i) \geq 2$, such that the collision times of all the $(i,i+1)$ collisions
in the interval $[t_{-(i,i+1)},t_{(i,i+1)}]$ are given by $s_{i,1}, \ldots, s_{i,n}$, 
with $s_{i,1} := t_{-(i,i+1)}$ and $s_{i,n} := t_{(i,i+1)}$. 
For $i = 1$, $n = n(1) \geq 0$, and by default $s_{1,0} = t_{-(1,2)}$, $s_{1,n+1} = t_{(1,2)}$.
For $i = 0$, $s_{0,1}, \ldots, s_{0,n}$, $n = n(0) \geq 1$, are simply the collision times 
of the lowest ball with the floor in the open interval $(t_{-(1,2)},t_{(1,2)})$.
We augment the collision time sequences by a first element $s_{i,0} := t_{-(1,2)}$
and a last element $s_{i,n+1} := t_{(1,2)}$, which yields the partitions
$\mathcal{P}_i = \bigcup_{k=0}^{n} [s_{i,k}, s_{i,k+1}]$, for every $i \in \{0,\ldots,N-1\}$. 

If at time $s_{i,k}$ of our partition, we face a singular collision, we might have to repartition. 
The details of this procedure are described in the last four paragraphs of the proof.

Assume that for every $\varepsilon > 0$ and every
$k \in \{0, \ldots, n+1\}$, where $n = n(1)$, we have
\begin{align}\label{assume1}
  v_{1}^{-}(s_{1,k}) - v_{2}^{-}(s_{1,k}) < \varepsilon,
\end{align}
that is, the velocity differences right before every $(1,2)$ collision
in $[t_{-(1,2)},t_{(1,2)}]$ are arbitrarily small. 

In order to apply (\ref{sandwichsum12}) we need to quantify how many $(2,3)$ collisions 
and floor returns of the lowest ball are in between two successful $(1,2)$ collisions. 
We introduce, for 
$i \in \{0,\ldots,N-1\}$, $j \in \{1,\ldots, N-1\}$, $k \in \{0,\ldots,n\}$, 
where $n = n(j)$, the functional
\begin{align}
  \mathtt{c}_{i}:\ &\mathcal{P}_j \to \mathbb{N}\nonumber\\
  &[s_{j,k}, s_{j,k+1}] \mapsto \mathtt{c}_{i}([s_{j,k}, s_{j,k+1}]) =: \mathtt{c}_{i,j,k}.\nonumber
\end{align}
The term $\mathtt{c}_{i,j,k}$ counts how many $(i,i+1)$ collisions appear in the interval
$[s_{j,k}, s_{j,k+1}]$ of the partition $\mathcal{P}_j$, i.e. in between two
successful $(j,j+1)$ collisions happening at time $s_{j,k}$ and $s_{j,k+1}$. Applying this
notation, we expand the velocity differences in (\ref{assume1}) backwards and according to
(\ref{sandwichsum12}) obtain for every $k \in \{1, \ldots, n(1)+1\}$
\begin{align}\label{sandwichsumapplied}
  \begin{array}{rcl}
  0 &\leq& v_{1}^{-}(s_{1,k}) - v_{2}^{-}(s_{1,k})\\[0.3cm]
  &=& 2\sum_{j = 1}^{\mathtt{c}_{0,1,k}}2jv_1^+(s_{0,g_0(j)})\\[0.3cm]
  &+& (1 - \gamma_{2})\sum_{l = 1}^{\mathtt{c}_{2,1,k}}
  \bigl(v_{2}^-(s_{2,g_2(l)}) - v_{3}^- (s_{2,g_2(l)})\bigr)\\[0.3cm]
	&+& 2\sqrt{(v_1^+(s_{1,k-1}))^2 + 2q_1(s_{1,k-1})} + v_1^+(s_{1,k-1}) - v_{2}^+(s_{1,k-1}),
  \end{array}
\end{align}
where the functions $g_0(j) \in \{1, \ldots, n(0)\}$ and $g_2(l) \in \{1, \ldots, n(2)\}$ 
enumerate the collision times subindices.
Using (\ref{sandwichsumapplied}) together with (\ref{pre}), (\ref{pre01}) and our assumption (\ref{assume1}),
implies for every $\varepsilon > 0$,
\begin{subequations}
	\begin{eqnarray}
  &&v_{2}^-(s_{2,k}) - v_{3}^-(s_{2,k}) < \varepsilon,\ \forall\ k \in \{1,\ldots,n(2)\},\label{implication1}
	\\[0.3cm]
	&&v_1^+(s_{0,k}) < \varepsilon,\ \forall\ k \in \{1,\ldots,n(0)\},\label{implication2}\\[0.3cm]
	&&v_1^+(s_{1,k-1}) < \varepsilon,\ q_1(s_{1,k-1}) < \varepsilon,\ \forall\ k \in \{1,\ldots,n(1)+1\}.
	\label{implication3}
	\end{eqnarray}
\end{subequations}
We repeat step (\ref{sandwichsumapplied}), by expanding the remaining 
velocity differences $v_{i}^-(s_{i,k}) - v_{i+1}^-(s_{i,k})$, for all 
$i \in \{2,\ldots,N-1\}$, $k \in \{2,\ldots,n(i)\}$ backwards. Using (\ref{pre}), (\ref{sandwichsum}), 
(\ref{implication1}), this leads to
\begin{align}\label{final1}
  v_{i}^-(s_{i,k}) - v_{i+1}^-(s_{i,k}) < \varepsilon,
\end{align}
for all $\varepsilon > 0$, $i \in \{2,\ldots,N-1\}$, $k \in \{2,\ldots,n(i)\}$. 
Every pre-collisional velocity difference $v_i^- - v_{i+1}^-$ occurring in 
$[t_{-(1,2)},s_{i,1})$ resp. $(s_{i,n},t_{(1,2)}]$ can be expanded 
forward resp. backward and by using (\ref{final1}) will be arbitrarily small 
as well. Therefore, (\ref{assume1}) implies that every ball to ball 
pre-collisional velocity difference in $[t_{-(1,2)},t_{(1,2)}]$ 
is arbitrarily small.

If the next ball to ball collision is $(i, i+1)$, $i \in \{1,\ldots,N-1\}$,
the collision time is given by
\begin{align}
\frac{q_{i+1} - q_{i}}{v_{i} - v_{i+1}}.\nonumber
\end{align}
If the denominator $v_{i} - v_{i+1}$ is arbitrarily small, $q_{i+1}$, $q_{i}$, has to be 
arbitrarily small as well, otherwise the collision time would be arbitrarily large, which would 
result in arbitrarily large velocities and contradict the finite energy assumption.
Since every velocity difference in $[t_{-(1,2)},t_{(1,2)}]$ is arbitrarily small, 
at time $t_{-(1,2)}$, all the balls are lying arbitrarily close to the floor with velocities 
being arbitrarily equal. 
Due to our construction, there is at least one $(0,1)$ collision in $[t_{-(1,2)},t_{(1,2)}]$.
Hence, the square root term in (\ref{sandwichsumapplied}) is present, which further implies 
(\ref{implication3}). Thus, at time $t_{-(1,2)}$, every ball lies arbitrarily close to the floor 
with arbitrarily small velocity. In this way, $H(q(t_0), p(t_0)) < \varepsilon$, for every $\varepsilon > 0$,
which means that our orbit would break through the constant energy surface.
Since this is impossible, we obtain a contradiction to our beginning assumption
(\ref{assume1}), hence, there exists a constant $C_1 > 0$
and at least one $k \in \{0, \ldots, n(1)+1\}$, such that
\begin{align}\label{final2}
  v_{1}^{-}(s_{1,k}) - v_{2}^{-}(s_{1,k}) \geq C_1.
\end{align}
In order to obtain the existence of a constant $C > 0$ and at least one 
$(i,i+1)$ collision, such that
$v_{i}^- - v_{i+1}^- \geq C$, for all  $i \in \{1, \ldots N-1\}$, we first pick the
previous resp. next $(2,3)$ collision before resp. after the $(1,2)$ collision 
in (\ref{final2}).
Let the past resp. future $(2,3)$ collision happen at $t_p$ resp. $t_f$. Using
(\ref{sandwichsum}) we expand $v_{2}^{-}(t_f) - v_{3}^{-}(t_f)$ backwards and
obtain
\begin{align}
	\begin{array}{rcl}
0 &<& v_{2}^{-}(t_f) - v_{3}^{-}(t_f)\\[0.3cm]
  &=& (1 + \gamma_{1})\sum_{j = 1}^{m}
  \bigl(v_{1}^-(r_j)- v_{2}^-(r_j)\bigr)\\[0.3cm]
  &+&(1 - \gamma_{3})\sum_{l = 1}^{n}
  \bigl(v_{3}^-(u_l) - v_{4}^- (u_l)\bigr)\\[0.3cm]	
	&+&  v_{2}^{+}(t_p) - v_{3}^{+}(t_p),\nonumber
	\end{array}
\end{align}
where $r_1,\ldots,r_m$ resp. $u_1,\ldots,u_n$ are the collision times of the $(1,2)$ 
resp. $(3,4)$ collisions in between the two $(2,3)$ collisions occurring at times $t_p$, $t_f$. 
Note, that the reason we denoted these collision times as $r_j$ resp. $u_l$ (and not $s_{1,j}$ resp. $s_{3,l}$)
is because one of the $(2,3)$ collisions may lie outside of $[t_{-(1,2)},t_{(1,2)}]$. This depends on the position 
of the $(1,2)$ collision at time $s_{1,k}$ from (\ref{final2}).

Assuming that both $(2,3)$ velocity differences in the past and future
are arbitrarily small yields a contradiction since
$v_{1}^{-}(s_{1,k}) - v_{2}^{-}(s_{1,k}) \geq C_1$. Hence, there exists 
a constant $C_2 > 0$, such that either $v_{2}^{-}(t_f) - v_{3}^{-}(t_f) \geq C_2$ or
$v_{2}^{-}(t_p) - v_{3}^{-}(t_p) \geq C_2$. Successfully continuing
this procedure we find positive constants $C_1, \ldots, C_{N-1} > 0$ and at least one 
$(i,i+1)$ collision, for all $i \in \{1,\ldots,N-1\}$, such that
\begin{align}\label{min}
  v_{i}^{-} - v_{i+1}^{-} \geq \min\{C_1,\ldots,C_{N-1}\}.
\end{align}
It follows from the way we obtained (\ref{min}), that the collision times of all 
$(i,i+1)$ collisions satisfying (\ref{min}) do not necessarily belong 
to $[t_{-(1,2)},t_{(1,2)}]$. 

The above steps can be repeated, thus, creating 
infinitely many compact intervals with a sequence of constant positive lower 
bounds for at least one $v_{1}^- - v_{2}^-$ per compact interval. This holds 
along every orbit. Those lower bounds have a global minimum, i.e.
\begin{align}
\min_{x\in\mathcal{M}}\ \min_{n\in\mathbb{N}}\ v_1^-(\mathtt{t}_n(x,1)) - v_2^-(\mathtt{t}_n(x,1))\nonumber
\end{align}
exists. Otherwise this would imply (\ref{final1}) and, hence, a contradiction. 
For this global lower bound we can repeat the steps from the last paragraph to obtain a global 
lower bound, say $C> 0$, for every pre-collisional velocity difference.

In the event of a singular collision between balls $i-1$, $i$, $i+1$, $i \geq 1$,
which happens at $s_{i-1,k} = s_{i,k}$, for some $k \in \{2,\ldots,n-1\}$, 
our orbit branches into two suborbits and the procedure above works for both branches, 
because there are further $(i-1,i)$, $(i+1,i)$ collisions flanking the singular 
collisions in the past and the future.
 
If the singularity occurs at the last possible collision time $s_{i-1,1} = s_{i,1}$ or
$s_{i-1,n} = s_{i,n}$, we have to repartition the collision times for one of the suborbits. 
We only outline $s_{i-1,1} = s_{i,1}$ since $s_{i-1,n} = s_{i,n}$ works in a similar way: 
If, for the first suborbit, the collision order is $(i-1,i) \to (i,i+1)$, nothing changes. If, 
$(i,i+1) \to (i-1,i)$, then we do not consider the $(i,i+1)$ collision but rather 
set $s_{i-1,1} = s_{i,1}$ to be the collision time of $(i-1,i)$. Then, continue as described 
in the beginning of the proof by picking the next collisions $(i,i+1), \ldots, (N-1,N)$ 
with corresponding (and possibly new) collision times $s_{i,1}, \ldots, s_{N-1,1}$. 

Note, that if $i = 1$ in the last paragraph we face no problem with either collision order 
$(0,1) \to (1,2)$, $(1,2) \to (0,1)$. In both cases we associate the collision time 
$s_{1,0} = t_{-(1,2)}$ with the $(1,2)$ collision (and in exactly the same manner, 
we associate $s_{1,n+1} = t_{(1,2)}$ either $(1,2)$ in the future).

The same procedure as in the last three paragraphs is initiated if the orbit experiences 
a singularity involving more than three balls.
\end{proof}

\begin{remark}
We want to bring to the readers attention, that it may be possible (depending on the dynamics), 
for some $x \in \mathcal{M}$, $i \in \{1,\ldots,N-1\}$, to have diverging collision time subsequences
$(\mathtt{u}_n)_{n\in\mathbb{N}}$, which satisfy, for instance
\begin{align}
\lim_{n\to\infty} v_i^-(\mathtt{u}_n(x,i)) - v_{i+1}^-(\mathtt{u}_n(x,i)) = 0.\nonumber
\end{align}
The importance is, that such a behaviour may only happen along a collision time subsequence, since 
there must be enough space left for $(\mathtt{t}_n)_{n\in\mathbb{N}}$ from 
Theorem \ref{heart} to exist.
\end{remark}

\section{The non-contraction property}\label{sec:noncon}

We begin this section by pointing out, that it is sufficient for the non-contraction property to hold if we 
only prove it for every $v \in \overline{\mathcal{C}(x)} \cap \partial B_{\|\cdot\|}(0,1)$, where 
$\partial B_{\|\cdot\|}(0,1)$ is the compact ball of unit radius, with respect to the norm $\|\cdot\|$, in 
tangent space.

Since the flow derivative between collisions is equal to the identity matrix, it is equivalent to formulate the 
non-contraction property in terms of the flow, i.e.
\begin{align}
\exists\ \zeta > 0,\ \forall\ t > 0 ,\ \forall\ x \in \mathcal{M} \setminus \mathcal{S}_t^+,\ \forall\ v \in 
\overline{\mathcal{C}(x)} \cap \partial B_{\|\cdot\|}(0,1):\ \|d_x \phi^{t} v\| \geq \zeta.\nonumber
\end{align}
We know \cite[Remark 10.3]{HT19}, that arbitrarily many $(0,1)$ collisions can occur in finite time. This is 
why we prefer to formulate the non-contraction property in terms of the flow, because we rather deal with 
finite times than arbitrarily many derivative map compositions.

Assume now that the strict unboundedness property (\nameref{strict}) holds for every point. 
We fix $E_0 > 0$ and define the function 
\begin{align}
	\tau_{E_0}^+:\ &\mathcal{M} \to \mathbb{R}_+,\nonumber\\
	&x \mapsto \tau_{E_0}^+(x),\nonumber
\end{align}
where
\begin{align}\label{tauplus}
	\tau_{E_0}^+(x) = \min\{t>0:\ Q(d_x \phi^t v) > E_0,\ \forall\ v \in \overline{\mathcal{C}(x)}\}.
\end{align}
The assumption of strict unboundedness (\nameref{strict}) together with the compactness of 
$\mathcal{M}$ and $\overline{\mathcal{C}(x)} \cap \partial B_{\|\cdot\|}(0,1)$ will 
help us to assert that $\tau_{E_0}^+$ is uniformly bounded from above, i.e.
\begin{align}\label{taubound}
\exists\ \mathtt{T} > 0, \forall\ x \in \mathcal{M}:\ \tau_{E_0}^+(x) \leq \mathtt{T}.
\end{align} 
This information is then utilized to split up the proof of the non-contraction 
property into two parts: First, we prove the non-contraction property for every collision of 
every feasible orbit in the fixed time interval $[0, \mathtt{T}]$ and, second, for every 
$t > \mathtt{T}$.

We begin with the proof of the uniform upper bound for $\tau_{E_0}^+$.
\begin{lemma}\label{taubound2}
	The function $\tau_{E_0}^+$ is uniformly bounded from above.
\end{lemma}
\begin{proof}
The assertion of strict unboundedness (\nameref{strict}) is equivalent to
\begin{align}
	\forall\ K \geq 0,\ \forall\ x \in \mathcal{M},\ \forall\ v \in 
\overline{\mathcal{C}(x)},\ \exists\ s_0 = s_0(K,x,v):\ Q(d_x \phi^t v) > K,\ \forall\ t \geq s_0.\nonumber
\end{align}
Since the $Q$-form is homogeneous (of degree two), the previous statement 
does not lose its general validity if we only assume it for  
$v \in \overline{\mathcal{C}(x)} \cap \partial B_{\|\cdot\|}(0,1)$.

We want to prove
\begin{align}
	\exists\ \mathtt{T} > 0,\ \forall\ x \in \mathcal{M},\ \forall\ v \in 
\overline{\mathcal{C}(x)} \cap \partial B_{\|\cdot\|}(0,1):\ s_0(E_0,x,v) \leq \mathtt{T}.\nonumber
\end{align}	
Assume on the contrary, that
\begin{align}
	\forall\ \mathtt{T} > 0,\ \exists\ x = x(\mathtt{T}) \in \mathcal{M},\ 
	\exists\ v = v(\mathtt{T})\in \overline{\mathcal{C}(x(\mathtt{T}))} \cap 
	\partial B_{\|\cdot\|}(0,1):\ s_0(E_0,x(\mathtt{T}),v(\mathtt{T})) > \mathtt{T}.\nonumber
\end{align}	
Due to compactness, the limits 
$\lim_{\mathtt{T} \to \infty} x(\mathtt{T}) = x_{\ast}$ resp. 
$\lim_{\mathtt{T} \to \infty} v(\mathtt{T}) = v_{\ast}$
lie in $\mathcal{M}$ resp. $\overline{\mathcal{C}(x_{\ast})} \cap \partial B_{\|\cdot\|}(0,1)$. 
In view of the role of $s_0$ in the strict unboundedness statement, our assumption implies
\begin{align}
	\exists\ x_{\ast} \in \mathcal{M},\ \exists\ v_{\ast} \in 
\overline{\mathcal{C}(x_{\ast})} \cap \partial B_{\|\cdot\|}(0,1):\ 
\lim_{t \to \infty} Q(d_{x_{\ast}} \phi^t v_{\ast}) \leq E_0,\nonumber
\end{align}
which clearly yields a contradiction to the strict unboundedness property.
\end{proof}
We introduce the norm
\begin{align}
\|(\delta \xi, \delta \eta)\|_{HT} := \|(\delta \xi, \delta \eta)\|_2 + \|\delta \eta\|_{CW},\nonumber
\end{align}
where
\begin{align}
\|\delta \eta\|_{CW}^2 = \sum_{i=1}^{N-2} \frac{(\delta \eta_{i+1} - \delta \eta_{i})^2}{m_i},\nonumber
\end{align}
is a norm on $\mathbb{R}^{N-1}$ introduced by Cheng and Wojtkowski in \cite[(11)]{ChW91}. It is invariant with respect to the submatrices 
$D_i$, $D_i^T$ of the ball to ball collision map derivatives given in (\ref{dxieta}). 
The norm $\| \cdot \|_2$ refers to the Euclidean norm.

We start with the first part of the proof by investigating how the fixed length of a 
vector changes, when it is subjected to floor or ball to ball collisions. 

\begin{lemma}
There exists a constant $E_1 > 0$, such that for all $i \in \{1,\ldots,N\}$, $x \in \mathcal{M}_{i,1}^+$, 
$(\delta \xi, \delta \eta) \in \overline{\mathcal{C}(x)} \cap \partial B_{\|\cdot\|_{HT}}(0,1)$ and $n \geq 1$, 
we have
\begin{align}
\|d_x \Phi_{0,1}^{n} (\delta \xi, \delta \eta)\|_{HT} \geq E_1.\nonumber
\end{align}
\end{lemma}

\begin{proof}
Using the definition of the floor derivative $d \Phi_{0,1}$ (\ref{dxieta}), we estimate
\begin{align}
\|d_x\Phi_{0,1}^n(\delta \xi, \delta \eta)\|_{HT} \geq \|(\delta \xi, nB\delta \xi + \delta \eta)\|_2 
\geq \max \{\|\delta \xi\|_2, \|nB\delta \xi + \delta \eta\|_2\}.\nonumber
\end{align}
We will be proving the following statement: There exists a constant $E_1 > 0$, such that for all $i \in \{1,
\ldots,N\}$, $x \in \mathcal{M}_{i,1}^+$, 
$(\delta \xi, \delta \eta) \in \overline{\mathcal{C}(x)} \cap \partial B_{\|\cdot\|_{HT}}(0,1)$ and $n \geq 1$, 
we have
\begin{align}\label{claim}
\|\delta \xi\|_2 \geq E_1\quad \vee \quad \|nB\delta \xi + \delta \eta\|_2 \geq E_1.
\end{align}
Assume on the contrary that the previous statement does not hold, i.e. for every $E_1 > 0$, 
there exists an $i \in \{1,\ldots,N\}$, $x \in \mathcal{M}_{i,1}^+$, 
$(\delta \xi, \delta \eta) \in \overline{\mathcal{C}(x)} \cap \partial B_{\|\cdot\|_{HT}}(0,1)$ and $n \geq 1$, 
such that 
\begin{align}\label{contradiction1}
\|\delta \xi\|_2 < E_1\quad \wedge \quad \|nB\delta \xi + \delta \eta\|_2 < E_1.
\end{align}
For $E_1$ sufficiently small, conditions (\ref{contradiction1}) imply
\begin{align}\label{contradiction2}
\lvert \delta \xi_1 \rvert, \lvert \delta \xi_2 \rvert, \ldots, \lvert \delta \xi_{N-1} \rvert < E_1,\quad
\lvert n\beta \delta \xi_1 + \delta \eta_1 \rvert, \lvert \delta \eta_2 \rvert,\ldots, \lvert \delta \eta_{N-1} \rvert < E_1.
\end{align}
Since $(\delta \xi, \delta \eta) \in \overline{\mathcal{C}(x)} \cap \partial B_{\|\cdot\|_{HT}}(0,1)$, (\ref{contradiction2}) implies that the length of the $\delta \eta$ component of the vector 
$(\delta \xi, \delta \eta)$ is concentrated on the first entry $\delta \eta_1$, i.e. there exists a constant $E_3 = E_3(E_1) > 0$, such that
\begin{align}\label{N2}
\lvert \delta \eta_1 \rvert \geq E_3.
\end{align}
If
\begin{align}
\delta \xi_1, \delta \eta_1 > 0,\quad \delta \xi_1, \delta \eta_1 < 0\,\quad \text{or}\quad 
\delta \xi_1 = 0,\nonumber
\end{align}
then
\begin{align}
\lvert n\beta \delta \xi_1 + \delta \eta_1 \rvert = 
n\beta \lvert \delta \xi_1 \rvert + \lvert \delta \eta_1 \rvert \geq 
\lvert \delta \eta_1 \rvert \geq E_3,\nonumber 
\end{align}
which contradicts (\ref{contradiction2}). Assume therefore that $\delta \xi_1 \delta \eta_1 < 0$.

Since the vector lies in the cone $\overline{\mathcal{C}(x)} \cap \partial B_{\|\cdot\|_{HT}}(0,1)$, $\sum_{i = 1}^{N-1} \delta \xi_i \delta \eta_i \geq 0$ 
must be satisfied. Combining this with the above yields
\begin{align}\label{secondcoming}
0 > \delta \xi_1 \delta \eta_1 \geq -(\delta \xi_2 \delta \eta_2 + \ldots + \delta \xi_{N-1} \delta \eta_{N-1}).
\end{align}
Due to (\ref{contradiction2}), (\ref{N2})
the second inequality in (\ref{secondcoming}) is violated, because the right hand side is of quadratic order 
$\mathcal{O}(E_1^2)$, while the $\delta \xi_1 \delta \eta_1$ term is of linear order $\mathcal{O}(E_1)$. 
Hence, for sufficiently small $E_1$, this implies 
$(\delta \xi, \delta \eta) \notin \overline{\mathcal{C}(x)} \cap \partial B_{\|\cdot\|_{HT}}(0,1)$, which contradicts 
assumption (\ref{contradiction1}) and yields our claim (\ref{claim}).
\end{proof}

A uniform lower bound for multiple ball to ball collisions can only be 
established for a fixed number of ball to ball collisions. Let $v_{\max} > 0$ be the largest possible velocity a ball can 
reach within the compact energy surface.
\begin{lemma}
For a fixed $\mathtt{n}_0 \geq 1$, let 
$d_x T^{\mathtt{n}_0}$ be a product of $\mathtt{n}_0$ ball to ball collision derivatives. Then, there exists a constant $E_2 > 0$, such that for all $m \geq 1$, 
$x \in \mathcal{M} \setminus (\bigcup_{i = 2}^{N}\mathcal{M}_{i,1}^+ \cup \mathcal{M}_{1,1}^{m,+})$, 
$(\delta \xi, \delta \eta) \in \overline{\mathcal{C}(x)} \cap \partial B_{\|\cdot\|_{HT}}(0,1)$, 
$n \leq \mathtt{n}_0$, we have
\begin{align}
\|d_x T^{n} (\delta \xi, \delta \eta)\|_{HT} \geq E_2.\nonumber
\end{align}
\end{lemma}
\begin{proof}
Using the ball to ball collision map derivatives (\ref{dxieta}), a first estimate gives,
\begin{align}
\|d_x T^n(\delta \xi, \delta \eta)\|_{HT} \geq 
\max \{\|D_n\delta \xi + U_n\delta \eta\|_2, \|D_n^T\delta \eta\|_{CW}\}.\nonumber
\end{align}
If $E_4 > 0$ and $\|\delta \eta \|_{CW} \geq E_4$, then the invariance with respect to $D_n^T$ 
of the norm $\|\cdot\|_{CW}$ immediately yields, that the vector is bounded from below. 

Assume therefore that $\|\delta \eta \|_{CW} < E_4$, for a value $E_4$, which will be 
chosen sufficiently small. Since 
$(\delta \xi, \delta \eta) \in \overline{\mathcal{C}(x)} \cap \partial B_{\|\cdot\|_{HT}}(0,1)$, 
there exists a constant $E_5 = E_5(E_4) > 0$, such that $\|\delta \xi\|_{CW} \geq E_5$. It is clear, that 
if $E_4$ decreases, $E_5$ increases. 

The matrix product $U_n$ is recursively defined by
\begin{align}
U_1 = F_{i_1},\ U_n = D_{i_n}U_{n-1} + F_{i_n}D_{n-1}^T,\nonumber
\end{align}
for some $i_1, \ldots, i_n \in \{1,\ldots,N\}$ depending on $x$ and $D_n = D_{i_n} \ldots D_{i_1}$. 
Repeatedly using the triangle inequality and the $D_i$-invariance of the $\|\cdot\|_{CW}$-norm, we estimate
\begin{align}
\|U_n\|_{CW} &\leq \|D_{i_n}\|_{CW}\|U_{n-1}\|_{CW} + \|F_{i_n}\|_{CW}\|D_{n-1}^T\|_{CW}\nonumber\\
&= \|U_{n-1}\|_{CW} + \|F_{i_n}\|_{CW}\nonumber\\
&\leq \|F_{i_1}\|_{CW} + \ldots + \|F_{i_n}\|_{CW}.\nonumber
\end{align} 
Remembering the definition of $F_{i_k}$ (\ref{dxieta}) and $\alpha_k$ (\ref{alpha}), we obtain the upper bound
\begin{align}
\|U_{n}\|_{CW} < n2v_{\max}\max_{i \in \{1,\ldots,N-1\}}\frac{2m_im_{i+1}(m_i - m_{i+1})}{(m_i + m_{i+1})^2}.\nonumber
\end{align}
We abbreviate
\begin{align}
	E_6 := \max_{i \in \{1,\ldots,N-1\}}\frac{2m_im_{i+1}(m_i - m_{i+1})}{(m_i + m_{i+1})^2},\nonumber
\end{align} 
and estimate
\begin{align}
\|D_n \delta \xi + U_n \delta \eta\| &\geq \|\delta \xi\|_2 - \|U_n\| \|\delta \eta\|_2 \nonumber\\
&\geq c_1\|\delta \xi\|_{CW} - c_2\|U_n\|_{CW} \|\delta \eta\|_{CW} \nonumber\\
&\geq c_1E_5 - c_2n 2 v_{\max}E_6E_4 \nonumber\\
&\geq c_1E_5 - c_2\mathtt{n}_0 2 v_{\max}E_6E_4,\nonumber
\end{align}
where $c_1, c_2 > 0$ are constants derived from the equivalence of norms. 

For the original statement to hold, we choose $E_4$ sufficiently small and obtain a lower bound $E_2 > 0$ for the last inequality.
\end{proof}

To conclude the first step of the non-contraction property it remains to prove
\begin{lemma}\label{collisionbound}
Let $T > 0$.  The number of ball to ball collisions in $[0,T]$ is bounded from above by a constant, which depends only on the 
length of the interval and the energy of the system.
\end{lemma}

\begin{proof}
We know from \cite{G78,G81,V79}, that the number of ball to ball collisions in $[0,T]$, without any floor interaction, is 
bounded from above. Assume therefore that we have arbitrarily many ball to ball collisions and floor interactions in $[0,T]$. 
Let $i_k \in \{0,\ldots,N-1\}$, $k \in \mathbb{N}$, and $(i_k, i_k + 1)$ be the aforementioned diverging collision sequence in 
decreasing order, i.e. $(i_k, i_k + 1)$ happens prior to $(i_{k-1}, i_{k-1} + 1)$. Additionally, we let each $(i_k, i_k + 1)$ 
happen at time $t_{(i_k, i_k + 1)}$. Due to the energy restriction it is clear that 
$v_{\max} > \lvert v_{i_1}^+(t_{(i_1, i_1 + 1)}) \rvert$. Using the formulas of the velocity time evolution (\ref{equations}) and elastic collisions (\ref{collisionqp}), (\ref{v1p1}), we expand $v_{i_1}^+(t_{(i_1, i_1 + 1)})$ $n$-times 
backward and obtain
\begin{align}\label{maxineq}
v_{\max} &> \lvert v_{i_1}^+(t_{(i_1, i_1 + 1)}) \rvert \nonumber\\
&= \lvert -(t_{(i_1, i_1 + 1)} - t_{(i_n, i_n + 1)}) + v_{i_n}^+(t_{(i_n, i_n + 1)}) + 
\sum_{k=1}^{c(n)} \gamma_{i_k} (v_{i_k}^- - v_{i_k + 1}^-) \rvert.
\end{align}
The positive integer function $c(n)$ counts how many ball to ball collisions happened within the $n$-backward iterations. 
Since we  also consider floor collisions $c(n) < n$, for $n$ large enough. Observe that each ball to ball collision adds a 
velocity difference term to the positive sum in (\ref{maxineq}). Hence, since we assume to have arbitrarily many ball to ball 
collisions, $\lim_{n \to \infty} c(n) = \infty$. Letting $n$ go to infinity in (\ref{maxineq}), we first obtain that 
$\lim_{n \to \infty} -(t_{(i_1, i_1 + 1)} - t_{(i_n, i_n + 1)})$ is bounded since $[0,T]$ is. Second, due to 
Theorem \ref{heart}, the sequence $(\gamma_{i_k} (v_{i_k}^- - v_{i_k + 1}^-))_{k\in\mathbb{N}}$ does not converge to 
zero and, thus, $\lim_{n \to \infty} \sum_{k=1}^{c(n)} \gamma_{i_k} (v_{i_k}^- - v_{i_k + 1}^-) = \infty$, which 
results in the contradiction $v_{\max} > \infty$.

We want to supplement the details for the reader, that the sum will be large enough for (\ref{maxineq}) to be violated 
after a uniform number of summations. The proof of this fact is similar to the proof of Lemma \ref{taubound2}. Abbreviate 
\begin{align}
a_{c(n)}(x) := \sum_{k=1}^{c(n)} \gamma_{i_k} (v_{i_k}^-(t_{(i_k, i_k + 1)}) - v_{i_k + 1}^-(t_{(i_k, i_k + 1)})).\nonumber
\end{align}
The divergence of $a_{c(n)}(x)$ is equivalent to
\begin{align}\label{condiv}
\forall\ K > 0,\ \forall\ x \in \mathcal{M},\ \exists\ M = M(K,x) \in \mathbb{N}:\ a_{c(n)}(x) > K,\ \forall\ n \geq M(K,x).
\end{align}
Due to the energy $c > 0$ there is a value $K_0 = K_0(c) > 0$ of $a_{c(n)}$ for which inequality (\ref{maxineq}) is 
violated. Observe that $K_0(c)$ additionally depends on the term of opposite sign  
$-(t_{(i_1, i_1 + 1)} - t_{(i_n, i_n + 1)})$ and, thus, on the length of $[0,T]$, hence, we have the depence $K_0 = K_0(c,T)$. 

We want to prove
\begin{align}
\forall\ c > 0,\ \forall\ T > 0,\ \exists\ M_0 = M_0(c,T) \in \mathbb{N},\ \forall\ x \in \mathcal{M}:\ 
M(K_0, x) < M_0(c,T).\nonumber
\end{align}
Assume on the contrary, that this does not hold, i.e.
\begin{align}
\exists\ c > 0,\ \exists\ T > 0,\ \forall\ M_0 \in \mathbb{N},\ \exists\ x = x(M_0):\ 
M(K_0, x(M_0)) \geq M_0.\nonumber
\end{align}
Since $\mathcal{M}$ is compact, the limit $\lim_{M_0 \to \infty} x(M_0) = x_{\ast}$ lies in $\mathcal{M}$. For this 
$x_{\ast}$, we obtain from (\ref{condiv})
\begin{align}
\lim_{n \to \infty} a_{c(n)}(x_{\ast}) \leq  K_0,\nonumber
\end{align}
which contradicts the divergence of $a_{c(n)}(x_{\ast})$. Therefore, the number of ball to ball collisions in $[0,T]$ are bounded by a constant $M_0(c,T)$, which depends only on the length $T$ of $[0,T]$ and  the energy $c > 0$ of the system.
\end{proof}
Combining the last three lemmas yields
\begin{corollary}
Let $T > 0$ and $c > 0$ be the energy of the system. Then, there exists a constant $\zeta_1 = \zeta_1(T,c) > 0$, such that 
the non-contraction property holds for every finite time interval $[0,T]$, i.e.
\begin{align}
\exists\ \zeta_1 > 0,\ \forall\ t \leq T,\ \forall\ x \in \mathcal{M} \setminus \mathcal{S}_t^+,\ \forall\ v \in 
\overline{\mathcal{C}(x)} \cap \partial B_{\|\cdot\|_{HT}}(0,1):\ \|d_x \phi^{t} v\| \geq \zeta_1.\nonumber
\end{align}
\end{corollary}
The corollary applies directly to the interval $[0,\mathtt{T}]$, where $\mathtt{T}$ is the uniform upper bound of $\tau_{E_0}^+$ (\ref{taubound}). We will now conclude the proof of the non-contraction property by proving it for all 
$t > \mathtt{T}$. 

Due to $\langle \delta \xi - \delta \eta, \delta \xi - \delta \eta \rangle \geq 0$, the Euclidean norm $\|\cdot\|_2$ and 
the $Q$-form can be related via
\begin{align}
\| (\delta \xi, \delta \eta) \|_2 \geq \sqrt{2} \sqrt{Q(\delta \xi, \delta \eta)}.\nonumber
\end{align}
Using the $Q$-monotonicity of the derivative and (\ref{tauplus}), (\ref{taubound}), we obtain
\begin{align}
\|d_x \phi^t(\delta \xi, \delta \eta)\|_{HT} \geq \|d_x \phi^t(\delta \xi, \delta \eta)\|_2 &\geq \sqrt{2} \sqrt{Q(d_x \phi^t(\delta \xi, \delta \eta))}\nonumber\\
&\geq \sqrt{2} \sqrt{Q(d_x \phi^{\mathtt{T}}(\delta \xi, \delta \eta))}\nonumber\\
&\geq \sqrt{2E_0}, \forall\ t > \mathtt{T}.\nonumber
\end{align}
This immediately proves the non-contraction property
\begin{corollary}
The non-contraction property formulated in terms of the norm $\|\cdot\|_{HT}$ holds with constant 
$\zeta = \min\{\zeta_1, \sqrt{2E_0}\}$.
\end{corollary}

\section{Ergodicity of a particle falling in a three dimensional wide wedge}\label{sec:wedge}
In \cite{W98}, Wojtkowski investigated the hyperbolicity of dynamical systems, which describe the motion 
of a particle subjected to constant acceleration in a variety of wedges.
We start by recapitulating the necessary prerequisites from \cite{W98,HT20} to prove our results. 
For a thorough introduction to the subject we recommend reading \cite{W98}.

The unrestricted configuration space $\mathbf{N}_q$ (\ref{configurationspace}) of $N$ falling balls
has the form of a wedge. Abbreviating $\textbf{q} = (q_1,\ldots,q_N)$, we can alternatively formulate it as
\begin{align}
W(b_1,\ldots,b_N)  = \Big\{\textbf{q} \in \mathbb{R}^N:\ \textbf{q} = \sum_{i=1}^N d_ib_i,\ d_i \in \mathbb{R}_+^N,\ i \in \{1,\ldots,N\}\Big\},\nonumber
\end{align}
where the set of linearly independent vectors $\{b_1,\ldots,b_N\}$, called \emph{generators}, are given by 
$b_i = (b_{i,k})_{k=1}^{N}$ with $b_{i,1} = \ldots = b_{i,i-1} = 0$, $b_{i,i} = \ldots = b_{i,N} = 1$.
Observe, that for every mass configuration $(m_1,\ldots,m_N)$, the system of falling balls has the same 
unrestricted configuration space. As before, we obtain the dynamics by intersecting the wedge with the energy 
surface $E_c$ (\ref{tangentenergy}).

We introduce the $Q$-isometric coordinate transformation
\begin{align}\label{fbtowedge}
x_i = \sqrt{m_i}q_i,\quad v_i = \frac{p_i}{\sqrt{m_i}}.
\end{align}
The newly obtained unit generators $\{e_1,\ldots,e_N\}$ in these coordinates become 
$\sqrt{M_i}e_i = (0,\ldots,0,\sqrt{m_i},\ldots,\sqrt{m_N})$, 
where $M_i = m_i + \ldots + m_N$. We observe that, up to a scalar multiple, every mass configuration defines a different 
wedge $W(e_1,\ldots,e_N)$, since the angles between the generators now depend on the masses. The inner product given by the kinetic energy in these coordinates is the standard scalar product $\langle \cdot, \cdot \rangle$ in $\mathbb{R}^N$. 
Since $\langle e_i, e_j \rangle = \sqrt{M_j}/\sqrt{M_i}$, it is easy to verify that 
\begin{align}\label{simple}
\begin{array}{lcl}
\langle e_i, e_{i+1} \rangle &>& 0,\ \forall\ i \in \{1,\ldots,N-1\},\\[0.2cm]
\langle e_i, e_j \rangle &=& \prod_{l = i}^{j-1} \langle e_l, e_{l+1} \rangle,\ \forall\ i,j \in \{1,\ldots,N\},\ 
i\neq j.
\end{array}
\end{align}
Wedges satisfying (\ref{simple}) are called \emph{simple} \cite[Proposition 2.3]{W98}.
  
In the three dimensional case, the hitting of the face $W(e_1,e_3)$, resp. $W(e_1,e_2)$ corresponds in the 
physical model to a $(1,2)$ resp. $(2,3)$ collision. The triple collision states are given by the intersection of the 
former faces, which  amounts to the first generator, i.e. $W(e_1,e_3) \cap W(e_1,e_2) = e_1$. 

It was shown in \cite{HT20} that for the mass restriction
\begin{align}\label{special}
2\sqrt{m_1m_3} = \sqrt{m_1 + m_2 + m_3},
\end{align}
the configuration wedge can be unfolded, by continuously reflecting it in the faces $W(e_1,e_2)$ resp. $W(e_2,e_3)$, into 
a \emph{wide} wedge \cite[Definition 6.1]{W98}. This wide wedge consists exactly of six simple 
wedges \cite[Figure 1]{HT20}. This idea is due to Wojtkowski and can be generalized to $N$ dimensions \cite{W16}. 

The triple collision states in the configuration space, which are represented by the first generator $e_1$, disappear in the wide wedge. More precisely, each trajectory which passes through the spot where $e_1$ was has a smooth continuation. Since the triple collision singularity manifold is the only obstacle to proper alignment \cite{LW92,HT20}, the system of a particle falling in the wide wedge, obtained for the special mass configuration (\ref{special}), satisfies the proper alignment condition \ref{C5}. 

For the Chernov-Sinai ansatz \ref{C4} to be valid, we need that the orbit for
$\mu_{\mathcal{S}^{\pm}}$-a.e. $x \in \mathcal{S}^{\pm}$ emerging from the singularity
manifold is strictly unbounded. This certainly holds since strict unboundedness is established for
every orbit in $\mathcal{M}$ \cite[Main Theorem]{HT20}.

As was proven earlier in this work, the non-contraction property \ref{C2} follows directly from 
the validity of the strict unboundedness for every point.

The Lagrangian subspaces $\mathbf{L}_1$, $\mathbf{L}_2$ (\ref{constant}) of the eventually strictly invariant 
cone field $\mathcal{C}$ are both constant in $\mathcal{M}$ and therefore continuous. This verifies 
condition \ref{C3}.

For the regularity of singularity manifolds \ref{C1}, we employ \cite[Lemma 7.7]{LW92}. The aforementioned lemma 
states that if $T:\ \mathcal{M} \setminus \mathcal{S}^+ \to \mathcal{M} \setminus \mathcal{S}^-$ is a 
diffeomorphism, the proper alignment condition \ref{C5} holds and $d_x T$ is $Q$-monotone for every 
$x \in \mathcal{M}$, then the regularity of singularity subsets follows. The last two conditions have 
already been affirmatively addressed. For the first one, observe that we outlined at the end of Section 
\ref{sec:fb}, that $T$ is a symplectomorphism up to and including the regular boundary $\mathcal{R}^+$ 
(\ref{symplectomorphism}).

Since the proper alignment condition \ref{C5} holds and every point is strictly unbounded \cite[Main Theorem]{HT20} 
it follows from Subsection \ref{sec:ab} that the set of sufficient points has measure one and is arcwise connected. 
Thus, the model of a particle falling in a three dimensional wedge is ergodic.

\subsection*{Acknowledgments}
I want to thank Nandor J. Sim\'anyi and Maciej P. Wojtkowski for their consultation and hospitality. 
Additionally my gratitude is expressed towards Marius N. Nkashama and UAB for providing the academic support, during the 
completion of this work in the summer of 2020.

\end{document}